\documentclass[12pt,a4paper]{amsart}
\usepackage{amssymb,color}
\usepackage[english]{babel}
\usepackage{verbatim,here}
\usepackage[T1]{fontenc}
\usepackage{epstopdf}
\usepackage{floatflt,graphicx}
\usepackage{color,xcolor}
\usepackage{enumerate}
\usepackage{animate}
\usepackage{a4wide}
\usepackage{amsmath,amssymb}
\usepackage{amsthm}

\usepackage{arydshln}

\graphicspath{{c:/manus09/nasser/180803/figures180803}{c:/manus09/nasser/180808/figures}}

\usepackage[autolinebreaks]{mcode}

\newcounter{minutes}
\divide\time by 60
\newcounter{hours}
\multiply\time by 60 \addtocounter{minutes}{-\time}
{\small
\curraddr{}
\email{mms.nasser@qu.edu.qa}
\curraddr{}
\email{vuorinen@utu.fi}
}

\dedicatory{}
\commby{}
\theoremstyle{plain}
\newtheorem{thm}[equation]{Theorem}

\newtheorem{lem}[equation]{Lemma}

\newtheorem{algorithm}[equation]{Algorithm}

\theoremstyle{definition}

\theoremstyle{remark}

\newtheorem{nonsec}[equation]{}

\numberwithin{equation}{section}

\newcommand{\beq}{\begin{equation}}
\newcommand{\eeq}{\end{equation}}
\newcommand{\ben}{\begin{enumerate}}
\newcommand{\een}{\end{enumerate}}
\newcommand{\bequu}{\begin{eqnarray*}}
\newcommand{\eequu}{\end{eqnarray*}}
\newcommand{\bequ}{\begin{eqnarray}}
\newcommand{\eequ}{\end{eqnarray}}

\renewcommand{\Im}{{\rm Im}\,}
\renewcommand{\Re}{{\rm Re}\,}

\newcommand\bigzero{\makebox(0,10){\text{\Huge O}}}
\newcommand\bigzeroL{\makebox(0,10){\text{\Large O}}}
\newcommand\bigzeroU{\makebox(0,0){\text{\Large O}}}

\newcommand{\bs}{{\bf s}}
\newcommand{\bt}{{\bf t}}

\DeclareMathOperator{\capp}{Cap}

\usepackage{url}
\usepackage{color}

\newcommand{\CC}{{\mathbb C}}

\renewcommand{\i}{\mathrm{i}}

\newcommand{\bI}{{\bf I}}
\newcommand{\bM}{{\bf M}}
\newcommand{\bN}{{\bf N}}
\newcommand{\bn}{{\bf n}}
\newcommand{\bT}{{\bf T}}
\newcommand{\capa}{\mathrm{cap}}

\renewcommand{\thefootnote}{\number_style{footnote}}
\begin{document}

\def\thefootnote{}

\title[Capacity of generalized condensers]{Numerical computation of the capacity of generalized condensers}

\author[M.M.S. Nasser]{Mohamed M S Nasser}
\address{Department of Mathematics, Statistics and Physics, Qatar University, P.O. Box 2713, Doha, Qatar.}
\author[M. Vuorinen]{Matti Vuorinen}
\address{Department of Mathematics and Statistics, University of Turku, Turku, Finland.}

\date{}

\begin{abstract}
We present a boundary integral method for numerical computation of the capacity of generalized condensers. 
The presented method applies to a wide variety of generalized condenser geometry including the cases when the plates of the generalized condenser are bordered by piecewise smooth Jordan curves or are rectilinear slits. The presented method is used also to compute the harmonic measure in multiply connected domains.
\end{abstract}

\maketitle

\footnotetext{\texttt{{\tiny File:~\jobname .tex, printed: \number\year-%
\number\month-\number\day, \thehours.\ifnum\theminutes<10{0}\fi\theminutes}}}
\makeatletter

\makeatother

%    Text of article.

% FILE:  gcsec103.tex
%%%%%%%%%%%%%%%%%%%%%%%%%%%%%%%%%%%%
%%%%%%%%%%%%%%%%%%%%%%%%%%%%%%%%%%%%%
%%%%%%%%%%%%%%%%%%%%%%%%%%%%%%%%%%%%%
\section{Introduction}
\label{sec:int}
%%%%%%%%%%%%%%%%%%%%%%%%
\setcounter{equation}{0}

The conformal capacity of condensers is an important notion in geometric function theory \cite{ah,avv,dek,du,kuh,ps,Vas02} and in various applications of electronics. However, the analytic forms of the capacity are known only for special types of condensers. So, the use of numerical methods for computing such capacity is unavoidable. Indeed, numerical computing of capacity of condensers have been intensively studied in the literature, see e.g.,~\cite{bbgghv,bsv,hrv1,hrv2,dek}. The capacity of condensers is one of the several ``conformal invariants'' which are powerful tools in complex analysis. Some of the other important examples of conformal invariants are the harmonic measure, the logarithmic capacity, the extremal length, the reduced extremal length, and the hyperbolic distance~\cite{ah,du,du2,gm,Vas02,vuo}. Numerical computing of such invariants has been studied also in the literature, see e.g.,~\cite{bbgghv,LSN17,ps,Ran2,RR}. 

Capacity of generalized condensers is another important example of conformal invariants~\cite{du,de,dk1,dk2,Vas02}. In this paper, we present a numerical method for computing the capacity of generalized condensers. We consider the case in which the plates of the generalized condensers are bordered by piecewise smooth Jordan curves or are rectilinear slits. As far as we know, the proposed method is the first numerical method for computing the numerical values of the capacity of the generalized condensers. The boundary integral equation with the generalized Neumann kernel~\cite{Nas-ETNA,Weg-Nas} plays a key role in developing our method. The presented method can be used also to compute the harmonic measure in multiply connected domains.

Let $B$ be an open subset of $\overline{\CC}=\CC\cup\{\infty\}$. We consider generalized condensers of the form $C=(B,E,\delta)$ where $E=\{E_k\}_{k=1}^m$, $m\ge2$, is a collection of nonempty closed pairwise disjoint sets in $E_k\subset B$ and $\delta=\{\delta_k\}_{k=1}^{m}$ is a collection of real numbers containing at least two different numbers. The set $G=B\backslash\cup_{k=1}^{m}E_k$ is called the field of the condenser $C$, the sets $E_k$ are the plates of the condenser, and the $\delta_k$ are the levels of the potential of the plates $E_k$, $k=1,2,\ldots,m$~\cite[p.~12]{du}. We assume that $G$ is a finitely connected domain without isolated boundary points and that $\partial G\cap\left(\cup_{k=1}^m E_k\right)$ consists of $m$ piecewise smooth Jordan curves, then the conformal capacity of $C$, $\capa(C)$, is given by the Dirichlet integral~\cite[p.~13, p.~305]{du}
\begin{equation}\label{eq:cap}
\capa(C)=\iint\limits_{G}|\nabla u|^2dxdy
\end{equation}
where $u$ is the potential function of the condenser $C$, i.e., $u$ is continuous in $\overline{G}$, harmonic in $G$, and equal to $\delta_k$ on $\partial E_k$ for $k=1,2,\ldots,m$ and satisfies $\partial u/\partial{\bf n}=0$ on $\partial B\backslash\cup_{k=1}^{m}E_k$ where $\partial u/\partial{\bf n}$ denotes the directional derivative of $u$ %differentiation 
along the outward normal.

The analytical description of the problem is given in Section 2 and it is based on the classical theory of integral equations \cite{Mik64} and on the definition of  the generalized capacity due to Dubinin \cite{du}. In Section 3 we formulate the computational problem as a Riemann-Hilbert problem and prove a preliminary analytical result. The main theoretical results are presented in Section 4 and they deal with unique solvability of algebraic linear systems related to the Riemann-Hilbert problem. Also an outline of an algorithm for the numerical solution of the integral equation is given. In Section 5 we give a MATLAB implementation of the algorithm. This algorithm is tested in Section 6 in the case of capacity computation of condensers with piecewise smooth boundary curves and results are compared, with good agreement of results, to earlier numerical results from \cite{bsv}. 
In Section 7 we apply the algorithm for the computation of the  capacity of generalized condensers. 
In Section 8, we use the presented algorithm with the help of conformal mappings to compute the capacity of rectilinear slit condensers. 
In the final Section 9 we show that the same method also works for the computation of the harmonic measure.

% FILE:  gcsec202.tex
%%%%%%%%%%%%%%%%%%%%%%%%%%%%%%%%%%%%
%%%%%%%%%%%%%%%%%%%%%%%%%%%%%%%%%%%%%

\section{The potential function}
\label{sec:u}

In this paper, for $k=1,2,\ldots,m$, we assume that $E_k=\overline{G_k}$ where $G_k$ is a simply connected domain bordered by a piecewise smooth Jordan curve $\Gamma_k$. We assume also that either $B=\CC$ or $B\subsetneq\CC$ is a multiply connected domain of connectivity $\ell\ge1$ bordered by $\ell$ piecewise smooth Jordan curves $\Gamma_k$ for $k=m+1,m+2,\ldots,m+\ell$. We assume $\ell=0$ when $B=\CC$ and $\infty\in B$ when $B$ is unbounded. Then, the field of the condenser is the multiply connected domain $G$ of connectivity $m+\ell$ bordered by
\[
\Gamma=\partial G=\bigcup_{k=1}^{m+\ell}\Gamma_k,
\]
where the orientation of the curves $\Gamma_k$ is such that $G$ is always on the left of $\Gamma_k$ for $k=1,2,\ldots,m+\ell$. For each $k=m+1,m+2,\ldots,m+\ell$, the simply connected domain on the right of $\Gamma_k$ will be denoted by $G_k$.

The domain $G$ is either bounded or unbounded. If $G$ is unbounded, we assume $\infty\in G$. If $G$ is bounded, then one of the simply connected domains $G_1,\ldots,G_m$ or $G_{m+1},\ldots,G_{m+\ell}$ is unbounded and contains $\infty$. If the unbounded domain is one of the domains $G_1,\ldots,G_m$, then we assume it is the domain $G_m$ and the curve $\Gamma_m$ enclose all the other curves $\Gamma_1,\ldots,\Gamma_{m-1},\Gamma_{m+1},\ldots,\Gamma_{m+\ell}$. Similarly, if the unbounded domain is one of the domains $G_{m+1},\ldots,G_{m+\ell}$, then we assume it is the domain $G_{m+\ell}$ and the curve $\Gamma_{m+\ell}$ enclose all the other curves $\Gamma_1,\ldots,\Gamma_{m+\ell-1}$.
Based on the boundedness of the domains $B$ and $G$, we define the integers $m'$ and $l'$ by
\begin{equation}\label{eq:mp}
m' =
\begin{cases}
m-1, & \mbox{if $G$ is bounded and $B$ is unbounded}, \\
m,   & \mbox{otherwise},
\end{cases}
\end{equation}
and
\begin{equation}\label{eq:ellp}
\ell' =
\begin{cases}
\ell-1, & \mbox{if $G$ is bounded and $B$ is bounded}, \\
\ell,   & \mbox{otherwise}.
\end{cases}
\end{equation}
In particular, if $G$ is unbounded, then $B$ is unbounded (since $G\subseteq B$), $m'=m$, $\ell'=\ell$, and hence $m'+\ell'=m+\ell$. If $G$ is bounded, then either $m'=m-1$ or $\ell'=\ell-1$ and hence $m'+\ell'=m+\ell-1$. Further, $m'=m-1$ means that $\Gamma_{m}$ is the external boundary component of $G$. Similarly, $\ell'=\ell-1$ means that the external boundary component of $G$ is $\Gamma_{m+\ell}$.
With these definitions of $m'$ and $\ell'$, the domains $G_1,\ldots,G_{m'}$ and $G_{m+1},\ldots,G_{m+\ell'}$ are bounded simply connected domains. For each of these bounded domains, we assume that $\alpha_k$ is an auxiliary point in $G_k$ for each $k=1,2,\ldots,m'$ and $\beta_k$ is an auxiliary point in $G_{m+k}$  for each $k=1,2,\ldots,\ell'$.

The potential function $u$ is then a solution of the Laplace equation $\Delta u=0$ with the mixed Dirichlet-Neumann boundary condition
\begin{subequations}\label{eq:mix-bd}
\begin{align}
\label{eq:mix-bd-1}
u(\zeta)&= \delta_k, \quad \zeta\in\Gamma_k, \quad k=1,2,\ldots,m, \\
\label{eq:mix-bd-2}
\frac{\partial u}{\partial\bn}(\zeta)&= 0, \quad \zeta\in\Gamma_k, \quad k=m+1,m+2,\ldots,m+\ell.
\end{align}
\end{subequations}
Note that the boundary value problem~\eqref{eq:mix-bd} reduces to a Dirichlet problem for $\ell=0$. Note also that the problem~\eqref{eq:mix-bd} does not reduce to a Neumann problem since $m\ge2$. The problem~\eqref{eq:mix-bd} has a unique solution $u$~\cite{IS19}.
 
A more general form of such mixed boundary value problem has been considered in~\cite{IS19} using a Cauchy integral method and in~\cite{Nas-bvp,Nas-jam} using the boundary integral equation with the generalized Neumann kernel. Due to the simple forms of the boundary conditions in~\eqref{eq:mix-bd}, the method presented in~\cite{Nas-bvp,Nas-jam} will be further simplified in this paper to obtain a simple, fast, and accurate method for computing the potential function $u$ and the capacity $\capa(C)$ of the generalized condenser $C$.

The harmonic function $u$ is the real part of an analytic function $F$ in $G$. The function $F$ is not necessarily single-valued, but it can be written as~\cite{Gak,gm,Mik64,Mus}
\begin{equation}\label{eq:F-u}
F(z)=g(z)-\sum_{k=1}^{m'} a_k\log(z-\alpha_k)-\sum_{k=1}^{\ell'} b_k\log(z-\beta_k)
\end{equation}
where $g$ is a single-valued analytic function in $G$ and $a_1,\ldots,a_{m'},b_1,\ldots,b_{\ell'}$ are undetermined real constants such that~\cite[\S31]{Mik64}
\begin{equation}\label{eq:ak}
a_k=\frac{1}{2\pi}\int_{\Gamma_k}\frac{\partial u}{\partial{\bf n}}ds, \quad k=1,2,\ldots,m',
\end{equation}
and
\[
b_k=\frac{1}{2\pi}\int_{\Gamma_{m+k}}\frac{\partial u}{\partial{\bf n}}ds, \quad k=1,2,\ldots,\ell'.
\]
Hence, using~\eqref{eq:mix-bd-2}, we have $b_k=0$ for all $k=1,2,\ldots,\ell'$.
Thus, the function $F$ has the representation
\begin{equation}\label{eq:F-u2}
F(z)=g(z)-\sum_{k=1}^{m'} a_k\log(z-\alpha_k).
\end{equation}
Since $u$ is harmonic in the domain $G$, then~\cite{Mik64}
\[
\int_{\Gamma}\frac{\partial u}{\partial{\bf n}}ds=0,
\]
which in view of~\eqref{eq:mix-bd-2} implies that
\begin{equation}\label{eq:u-k}
\sum_{k=1}^{m}\int_{\Gamma_k}\frac{\partial u}{\partial{\bf n}}ds=0.
\end{equation}

Recall that $a_1,\ldots,a_{m'}$ are given in~\eqref{eq:ak}. So, if $m'=m-1$, we define
\begin{equation}\label{eq:am}
a_m=\frac{1}{2\pi}\int_{\Gamma_m}\frac{\partial u}{\partial{\bf n}}ds.
\end{equation}
Hence, it follows from~\eqref{eq:ak}, \eqref{eq:u-k}, and~\eqref{eq:am} that
\begin{equation}\label{eq:sum-ak}
\sum_{k=1}^{m}a_k=\sum_{k=1}^{m}\frac{1}{2\pi}\int_{\Gamma_k}\frac{\partial u}{\partial{\bf n}}ds = 0,
\end{equation}
which implies, in the case $m'=m-1$, that
\begin{equation}\label{eq:am-b}
a_m=-\sum_{k=1}^{m-1}a_k.
\end{equation}

Using Green's formula~\cite[p.~4]{du}, Equation~\eqref{eq:cap} can be written as
\begin{equation}\label{eq:cap-n}
\capa(C)=\int_{\partial G}u\frac{\partial u}{\partial{\bf n}}ds.
\end{equation}
Since $\partial u/\partial\bn=0$ on $\partial B=\cup_{k=1}^{\ell}\Gamma_{m+k}$ and $u=\delta_k$ on $\Gamma_k$ for $k=1,2,\ldots,m$, then in view of~\eqref{eq:ak} and~\eqref{eq:am}, we have
\begin{equation}\label{eq:cap-k}
\capa(C)=\sum_{k=1}^{m}\delta_k\int_{\Gamma_k}\frac{\partial u}{\partial{\bf n}}ds
=2\pi\sum_{k=1}^{m}\delta_ka_k.
\end{equation}
Equation~\eqref{eq:cap-k} gives us a simple formula for computing the capacity of the generalized condenser $C$ in terms of the levels $\delta_k$ of the potential of the plates and the values of the constants $a_k$ for $k=1,2,\ldots,m$. 

In this paper, the boundary integral equation with the generalized Neumann kernel will be used to compute the constants $a_k$ as well as the values of the function $u(z)$ for $z\in G$. However, to use the integral equation, we will first reformulate the above mixed boundary value problem as a Riemann-Hilbert problem as it will be described in the next section. Solving the mixed boundary value problem by reducing it to a Riemann-Hilbert problem is a well known approach and has been used by many researchers in the literature (see e.g.,~\cite{Nas-bvp,Gak,Haas,Mus,Nas-jam}).

% FILE:  gcsec302.tex
%%%%%%%%%%%%%%%%%%%%%%%%%%%%%%%%%%%%
%%%%%%%%%%%%%%%%%%%%%%%%%%%%%%%%%%%%
\section{The Riemann-Hilbert problem}

For each $k=1,2,\ldots,m+\ell$, the boundary component $\Gamma_k$ is parametrized by a $2\pi$-periodic complex function $\eta_k(t)$, $t\in J_k:=[0,2\pi]$. The total parameter domain $J$ is the disjoint union of the $m+\ell$ intervals $J_1,\ldots,J_{m+\ell}$, 
\[
J = \bigsqcup_{k=1}^{m+\ell} J_k=\bigcup_{k=1}^{m+\ell}\{(t,k)\;:\;t\in J_k\}.
\]
The elements of $J$ are ordered pairs $(t,k)$ where $k$ is an auxiliary index indicating which of the intervals contains the point $t$~\cite{Nas-ETNA}. 
A parametrization of the whole boundary $\Gamma$ is then defined by
\begin{equation}\label{e:eta-1}
\eta(t,k)=\eta_k(t), \quad t\in J_k,\quad k=1,2,\ldots,m+\ell.
\end{equation}
For a given $t$, the value of auxiliary index $k$ such that $t\in J_k$ will be always clear from the context. So we replace the pair $(t,k)$ in the left-hand side of~(\ref{e:eta-1}) by $t$ in the same way as in~\cite{Nas-ETNA}. Thus, the function $\eta$ in~(\ref{e:eta-1}) is written as
\begin{equation}\label{e:eta}
\eta(t)= \left\{ \begin{array}{l@{\hspace{0.5cm}}l}
\eta_1(t),&t\in J_1,\\
\eta_2(t),&t\in J_2,\\
\hspace{0.3cm}\vdots\\
\eta_{m+\ell}(t),&t\in J_{m+\ell}.
\end{array}
\right.
\end{equation}

Since $u=\delta_k$ is known on the boundary components $\Gamma_k$ for $k=1,2,\ldots,m$ and since $u=\Re F$, then the boundary values of the function $F$ satisfy
\begin{equation}\label{eq:F-bc-1}
\Re[F(\eta(t))]=\delta_k,\quad \eta(t)\in\Gamma_k, \quad k=1,2,\ldots,m.
\end{equation}
On the boundaries $\partial B=\cup_{k=1}^{\ell}\Gamma_{m+k}$, the potential function $u$ satisfies the boundary condition $\partial u/\partial\bn=0$ where $\bn$ is the outward normal vector on $\partial B$. Let $\bT$ be the unit tangent vector on $\partial B$. Then, for $\eta(t)\in\partial B$,
\begin{equation}\label{eq:tangent}
\bn(\eta(t))=-\i\bT(\eta(t))=-\i\frac{\eta'(t)}{|\eta'(t)|}=e^{\i\nu(\eta(t))}
\end{equation}
where $\nu(\eta(t))$ is the angle between the positive real axis and the normal vector $\bn(\eta(t))$.
Using the Cauchy-Riemann equations, the derivative of the analytic function $F$ is then $F'(z)=\frac{\partial u(z)}{\partial x}-\i\frac{\partial u(z)}{\partial y}$. Thus,
\begin{equation}\label{eq:0-par-u}
\frac{\partial u}{\partial\bn}=\nabla u\cdot\bn
=\cos(\nu)\frac{\partial u}{\partial x}+\sin(\nu)\frac{\partial u}{\partial y}
=\Re\!\left[e^{\i\nu}\left(\frac{\partial u}{\partial x}-\i\frac{\partial u}{\partial y}\right)\right]\!
=\Re\!\left[\frac{-\i\eta'(t)}{|\eta'(t)|}F'(\eta(t))\right]\!
\end{equation}
which, in view of~\eqref{eq:mix-bd-2}, implies that
\[
\Re\left[-\i\eta'(t)F'(\eta(t))\right]=0, \quad \eta(t)\in\Gamma_{m+k}, \quad k=1,2,\ldots,\ell.
\]
Integrating with respect to the parameter $t$ yields
\begin{equation}\label{eq:F-bc-2}
\Re[-\i F(\eta(t))]=\nu_k,\quad \eta(t)\in \Gamma_{m+k}, \quad k=1,2,\ldots,\ell,
\end{equation}
where $\nu_1,\nu_2,\ldots,\nu_\ell$ are real constants of integration. Thus, by~\eqref{eq:F-bc-1} and~\eqref{eq:F-bc-2}, the boundary values of the function $F$ satisfy the boundary condition
\[
\Re\left[e^{-\i\theta(t)}F(\eta(t))\right]=\delta(t)+\nu(t)
\]
where
\begin{equation}\label{eq:theta}
\theta(t) =
\begin{cases}
0, & t \in J_1, \\
 \vdots & \\
0, & t \in J_m, \\
\pi/2, & t \in J_{m+1}, \\
\vdots & \\
\pi/2, & t \in J_{m+\ell},
\end{cases}, \quad
\delta(t) =
\begin{cases}
\delta_1, & t \in J_1, \\
 \vdots & \\
\delta_m, & t \in J_m, \\
0, & t \in J_{m+1}, \\
\vdots & \\
0, & t \in J_{m+\ell},
\end{cases}, \quad
\nu(t) =
\begin{cases}
0, & t \in J_1, \\
 \vdots & \\
0, & t \in J_m, \\
\nu_1, & t \in J_{m+1}, \\
\vdots & \\
\nu_\ell, & t \in J_{m+\ell},
\end{cases}
\end{equation}
i.e., $\theta(t)=0$ and $\nu(t)=0$ for $\ell=0$.
Then, it follows from~\eqref{eq:F-u2} that the single-valued analytic function $g$ satisfies the boundary condition
\begin{equation}\label{eq:get1}
\Re\left[e^{-\i\theta(t)}g(\eta(t))\right]=\delta(t)+\nu(t)+\sum_{k=1}^{m'} a_k\Re\left[e^{-\i\theta(t)}\log(\eta(t)-\alpha_k)\right].
\end{equation}

\begin{lem}\label{lem:gam-k}
The functions $\gamma_k$, for $k=1,\ldots,m'$, defined on $J$ by
\begin{equation}\label{eq:gam-k}
\gamma_k(t) = \begin{cases}
\displaystyle
\Re\left[e^{-\i\theta(t)}\log(\eta(t)-\alpha_k)\right], 
& {\rm if\;} \ell'=\ell, \\[10pt]
\displaystyle
\Re\left[e^{-\i\theta(t)}\log\frac{\eta(t)-\alpha_k}{\eta(t)-\alpha}\right], 
& {\rm if\;} \ell'=\ell-1,
\end{cases}
\end{equation}
are periodic for $t\in J_j$, $j=1,2,\ldots,m+\ell$.
For both cases, we have
\begin{equation}\label{eq:sum-gam-k}
\sum_{k=1}^{m'} a_k\gamma_k(t)=\sum_{k=1}^{m'} a_k\Re\left[e^{-\i\theta(t)}\log(\eta(t)-\alpha_k)\right].
\end{equation}
\end{lem}
\begin{proof}
Since $\theta(t)=0$ when $t\in J_j$ for each $j=1,2,\ldots,m$, then the functions $\gamma_k(t)$ in~\eqref{eq:gam-k} are periodic for $t\in J_j$ for each $j=1,2,\ldots,m$.

When $t\in J_j$ for each $j=m+1,m+2,\ldots,m+\ell$, we have the following two cases:

a) $\ell'=\ell$.
For this case, $\Gamma_{m+\ell}$ is not the external boundary component of $G$.
Recall that, for each $k=1,2,\ldots,m'$, $\alpha_k$ is in the interior of the curve $\Gamma_k$. Thus, none of the auxiliary points $\alpha_1,\ldots,\alpha_{m'}$ is interior to any of the curves $\Gamma_{m+1},\ldots,\Gamma_{m+\ell}$. Hence, the winding number of the function $z-\alpha_k$ is always zero along each boundary component $\Gamma_{m+k}$ for $k=1,2,\ldots,\ell$. Thus, we can always choose a branch cut of the logarithm function such that the functions $\gamma_k(t)$ given by the first formula in~\eqref{eq:gam-k} are periodic for $t\in J_j$ for each $j=m+1,m+2,\ldots,m+\ell$.

b)  $\ell'=\ell-1$.
For this case, $\Gamma_{m+\ell}$ is the external boundary component of $G$. Hence, none of the auxiliary points $\alpha,\alpha_1,\ldots,\alpha_{m'}$ is interior to any of the curves $\Gamma_{m+1},\ldots,\Gamma_{m+\ell-1}$. However, all the auxiliary points $\alpha,\alpha_1,\ldots,\alpha_{m'}$ are interior to the curve $\Gamma_{m+\ell}$. Thus, the winding number of the function $\frac{z-\alpha_k}{z-\alpha}$ is always zero along each boundary component $\Gamma_{m+k}$ for $k=1,2,\ldots,\ell$. Hence, we can choose a branch cut of the logarithm function such that the functions $\gamma_k(t)$ given by the second formula in~\eqref{eq:gam-k} are periodic for $t\in J_j$ for each $j=m+1,m+2,\ldots,m+\ell$. For this case, we need to prove also that equation~\eqref{eq:sum-gam-k} holds for the functions $\gamma_k(t)$ defined by the second formula in~\eqref{eq:gam-k}. Since $\Gamma_{m+\ell}$ is the external boundary component of $G$, we have $m'=m$, and by~\eqref{eq:sum-ak}, we have $\sum_{k=1}^{m'}a_k=0$. Thus,
\begin{eqnarray*}
\sum_{k=1}^{m'} a_k\gamma_k(t)
&=&\sum_{k=1}^{m'} a_k\Re\left[e^{-\i\theta(t)}\log\frac{\eta(t)-\alpha_k}{\eta(t)-\alpha}\right]\\
&=&\Re\left[e^{-\i\theta(t)}\log(\eta(t)-\alpha)\right]\sum_{k=1}^{m'} a_k+\sum_{k=1}^{m'} a_k\Re\left[e^{-\i\theta(t)}\log\frac{\eta(t)-\alpha_k}{\eta(t)-\alpha}\right]\\
&=&\sum_{k=1}^{m'} a_k\Re\left[e^{-\i\theta(t)}\log(\eta(t)-\alpha)+e^{-\i\theta(t)}\log\frac{\eta(t)-\alpha_k}{\eta(t)-\alpha}\right]\\
&=&\sum_{k=1}^{m'} a_k\Re\left[e^{-\i\theta(t)}\log(\eta(t)-\alpha_k)\right],
\end{eqnarray*}
and hence~\eqref{eq:sum-gam-k} holds for the functions $\gamma_k(t)$ defined by the second formula in~\eqref{eq:gam-k}.
\end{proof}

Taking into account~\eqref{eq:sum-gam-k}, we rewrite the boundary condition~\eqref{eq:get1} as
\begin{equation}\label{eq:get}
\Re\left[e^{-\i\theta(t)}g(\eta(t))\right]=\delta(t)+\nu(t)+\sum_{k=1}^{m'} a_k\gamma_k(t)
\end{equation}
where the functions $\gamma_k$ are defined by~\eqref{eq:gam-k}.
Since we are interesting in computing only $u=\Re F$, we can assume that $g(\infty)=c$ is real for unbounded $G$ and $g(\alpha)=c$ is real for bounded $G$. We introduce an auxiliary function $f$ defined in $G$ by
\begin{equation}\label{eq:f-g}
f(z) = \begin{cases}
g(z)-c, & {\rm if\;} G \;{\rm is\; unbounded}, \\
(g(z)-c)/(z-\alpha), & {\rm if\;} G \;{\rm is\; bounded}.
\end{cases}
\end{equation}
Then $f$ is a single-valued analytic function in $G$ with $f(\infty)=0$ for unbounded $G$.
Let $A(t)$ be the complex-valued function defined by~\cite{Nas-ETNA}
\begin{equation}\label{eq:A}
A(t) = \begin{cases}
e^{-\i\theta(t)}, & {\rm if\;} G \;{\rm is\; unbounded}, \\
e^{-\i\theta(t)}(\eta(t)-\alpha), & {\rm if\;} G \;{\rm is\; bounded}.
\end{cases}
\end{equation}
Hence the boundary condition~\eqref{eq:get} implies that the function $f$ is a solution of the following Riemann-Hilbert problem
\begin{equation}\label{eq:rhp-f}
\Re\left[A(t)f(\eta(t))\right]=-c\cos\theta(t)+\delta(t)+\nu(t)+\sum_{k=1}^{m'} a_k\gamma_k(t).
\end{equation}

Observe that solving the Riemann-Hilbert problem~\eqref{eq:rhp-f} requires finding the unknown analytic functions $f$ as well as the unknown real constants $a_1,\ldots,a_{m},c,\nu_1,\ldots,\nu_\ell$ in the right-hand side of~\eqref{eq:rhp-f}.

% FILE:  gcsec402.tex
%%%%%%%%%%%%%%%%%%%%%%%%%%%%%%%%%%%%
%%%%%%%%%%%%%%%%%%%%%%%%%%%%%%%%%%%%
%----------------------------------------
\section{The generalized Neumann kernel}

The generalized Neumann kernel $N(s,t)$ is defined for $(s,t)\in J\times J$ by~\cite{Weg-Nas}
\begin{equation*}
N(s,t)=
\frac{1}{\pi}\Im\left(\frac{A(s)}{A(t)}\frac{\dot\eta(t)}{\eta(t)-\eta(s)}\right).
\end{equation*}
Closely related to the kernel $N$ is the following kernel $M(s,t)$ defined for $(s,t)\in J\times J$ by~\cite{Weg-Nas}
\begin{equation*}
M(s,t) =
\frac{1}{\pi}\Re\left(\frac{A(s)}{A(t)}\frac{\dot\eta(t)}{\eta(t)-\eta(s)}\right).
\end{equation*}
The kernel $N$ is continuous and the kernel $M$ has a singularity of cotangent type~\cite{Weg-Nas}.

Let $H$ denote the space of all real-valued H\"older continuous functions on the boundary $\Gamma$. In this paper, for simplicity, if $\phi$ is a real-valued function defined on the boundary $\Gamma$, then we write $\phi(\eta(t))$ as $\phi(t)$. Further, any piecewise constant function $h\in H$ defined by
\begin{equation*}%\label{e:h-1}
h(t)=h_k\quad {\rm for}\quad t\in J_k,
\end{equation*}
with real constants $h_k$ for $k=1,\ldots,m+\ell$ will be denoted by
\begin{equation*}%\label{e:h-2}
h(t)=(h_1,\ldots,h_{m+\ell}),\quad t\in J.
\end{equation*}
The integral operators with the kernels $N(s,t)$ and $M(s,t)$ are defined on $H$ by
\begin{eqnarray}
\label{eq:bN}
(\bN \phi)(s) &=& \int_J N(s,t) \phi(t) \, dt, \quad s \in J, \\
\label{eq:bM}
(\bM \phi)(s) &=& \int_J M(s,t) \phi(t) \, dt, \quad s \in J.
\end{eqnarray}
The identity operator on $H$ will be denoted by $\bI$. Then, we have the following theorem from~\cite{Nas-JMAA1}.

\begin{thm}\label{thm:ie}
For each $k=1,2,\ldots,m'$, let the function $\gamma_k$ be given by~\eqref{eq:gam-k}. Then, there exists a unique real-valued function $\mu_k\in H$ and a unique piecewise constant real-valued function $h_k=(h_{1,k},h_{2,k},\ldots,h_{m+\ell,k})$ such that
\begin{equation}\label{eq:fet-k}
A(t)f_k(\eta(t))=\gamma_k(t)+h_k(t)+\i\mu_k(t), \quad t \in J,
\end{equation}
are boundary values of an analytic function $f_k$ in $G$ with $f(\infty)=0$ for unbounded $G$. The function $\mu_k$ is the unique solution of the integral equation
\begin{equation}\label{eq:ie-k}
(\bI-\bN)\mu_k=-\bM\gamma_k
\end{equation}
and the function $h_k$ is given by
\begin{equation}\label{eq:h-k}
h_k=[\bM\mu_k-(\bI-\bN)\gamma_k]/2.
\end{equation}
\end{thm}

The integral equation~\eqref{eq:ie-k} been used for computing the conformal map from bounded and unbounded multiply connected domains onto several canonical slit domains, see e.g.,~\cite{Nas-Siam1,Nas-JMAA1,Nas-ETNA}.
The following lemma is needed to prove Theorems~\ref{thm:c1} and~\ref{thm:c2} below.

\begin{lem}\label{lem:solv}
If $f$ is an analytic function in $G$ with $f(\infty)=0$ for unbounded $G$ such that its boundary values satisfy the boundary condition
\begin{equation}\label{eq:rhp-g}
\Re[A(t)f(\eta(t))]=\gamma(t)
\end{equation}
for a piecewise constant real-valued function $\gamma(t)=(c_{1},c_{2},\ldots,c_{m+\ell})$, then $f$ is the zero function and $c_{1}=c_{2}=\cdots=c_{m+\ell}=0$.
\end{lem}
\begin{proof}
The solvability of the Riemann-Hilbert problem~\eqref{eq:rhp-g} depends on the winding number of the function $A$. For the function $A$ defined in~(\ref{eq:A}), the Riemann-Hilbert problem~\eqref{eq:rhp-g} is not necessarily solvable~\cite{Nas-ETNA}. However, by Theorem~\ref{thm:ie}, a unique piecewise constant real-valued function $h(t)=(h_{1},h_{2},\ldots,h_{m+\ell})$ exists such that the Riemann-Hilbert problem
\[
\Re[A(t)f(\eta(t))]=\gamma(t)+h(t)
\]
is uniquely solvable (see also~\cite{Nas-ETNA,Weg-Nas}). By the uniqueness of the piecewise constant function $h$ and since the function $\gamma$ is a piecewise constant function, the function $h$ must be given by $h(t)=-\gamma(t)$ since the problem
\[
\Re[A(t)f(\eta(t))]=\gamma(t)+h(t)=0
\]
will be solvable and has the zero solution $f(z)=0$.
\end{proof}

In the remaining part of this section, we shall use Theorem~\ref{thm:ie} to present a method for computing the real constants $a_1,\ldots,a_m$ and hence computing $\capa(C)$ through~\eqref{eq:cap-k}. Recall from~\eqref{eq:mp} that either $m'=m$ or $m'=m-1$. These two cases of $m'$ will be considered separately in the following two subsections.  

%----------------------------------------
\subsection{Case I: $m'=m$}

This case includes the following two subcases:
\begin{enumerate}
	\item Both $G$ and $B$ are unbounded (see Figure~\ref{fig:gc-c1-1}). For this subcase, we have $m'=m\ge2$, $\ell'=\ell\ge0$ (where $B=\CC$ for $\ell=0$), $A$ is given by the first formula in~\eqref{eq:A}, and the functions $\gamma_k$ for $k=1,2,\ldots,m$ are given by the first formula in~\eqref{eq:gam-k}.
	\item Both $G$ and $B$ are bounded (see Figure~\ref{fig:gc-c1-2}).  For this subcase, we have $m'=m\ge2$, $\ell'=\ell-1\ge0$, $\Gamma_{m+\ell}$ is the external boundary component of $G$, $A$ is given by the second formula in~\eqref{eq:A}, and the functions $\gamma_k$ for $k=1,2,\ldots,m$ are given by the second formula in~\eqref{eq:gam-k}.
\end{enumerate}
For these two subcases, all the simply connected domains $G_1,\ldots,G_{m}$ are bounded (see Figures~\ref{fig:gc-c1-1} and~\ref{fig:gc-c1-2}). In Figures~\ref{fig:gc-c1-1} and~\ref{fig:gc-c1-2}, and in all figures throughout the paper, the boundaries of the domain $B$ are the ``dash-dotted'' curves and the boundaries of the plates of the condenser are the ``solid'' curves.

\begin{figure}[ht] %
\centerline{
\scalebox{0.5}{\includegraphics[trim=0 0 0 0,clip]{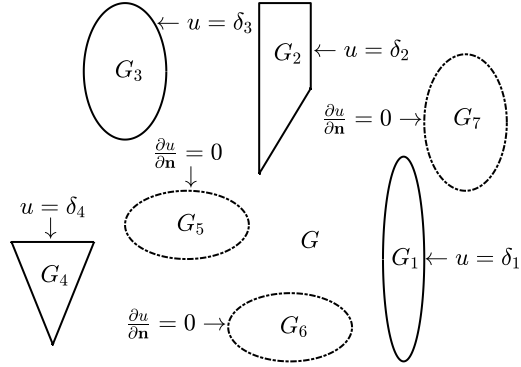}}
}
\caption{An example of an unbounded multiply connected domain $G$ for $m=4$ and $\ell=3$ for Case I (both $G$ and $B$ are unbounded).}
\label{fig:gc-c1-1}
\end{figure}

\begin{figure}[ht] %
\centerline{
\scalebox{0.5}{\includegraphics[trim=0 0 0 0,clip]{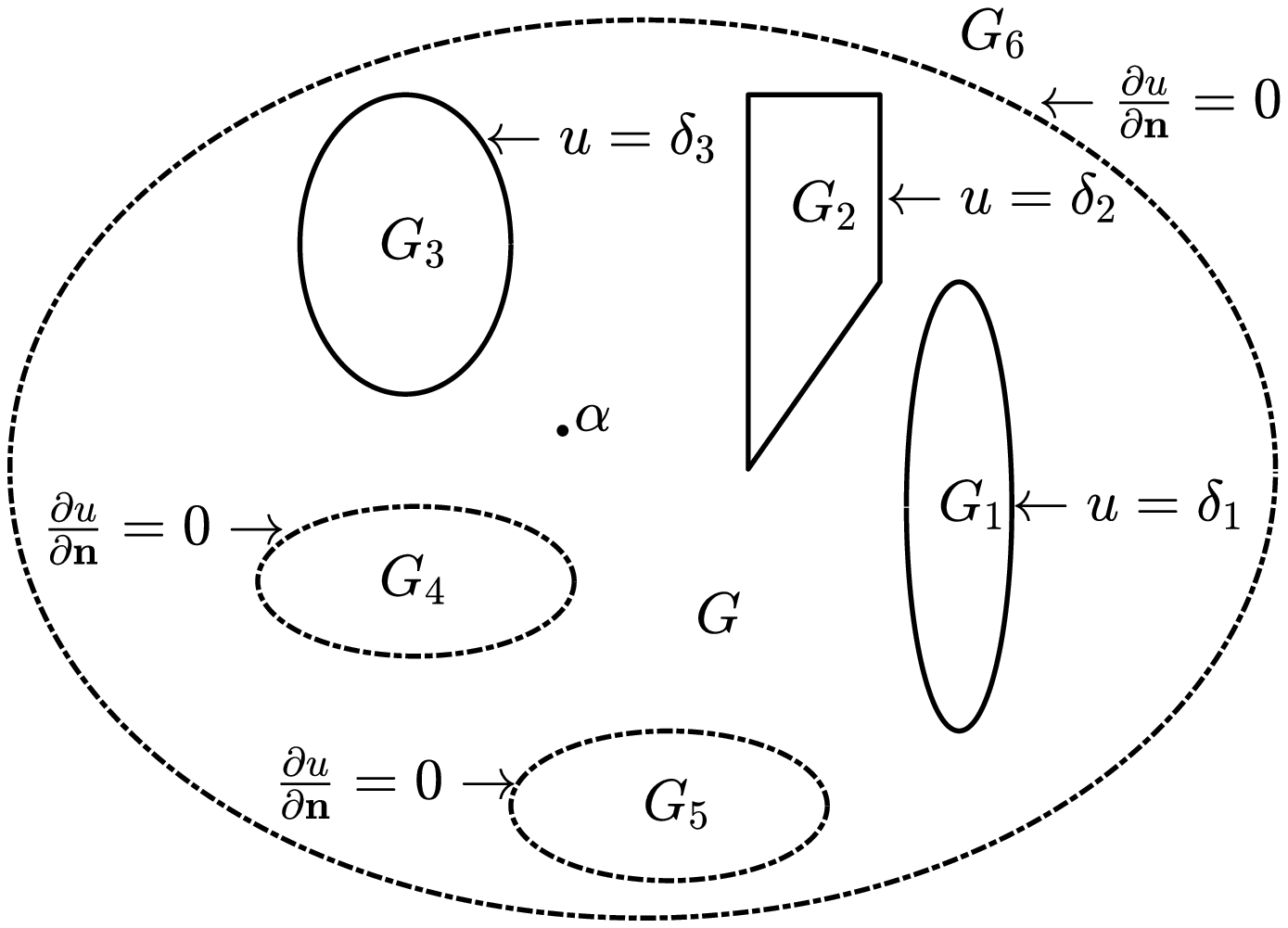}}
}
\caption{An example of a bounded multiply connected domain $G$ for $m=3$ and $\ell=3$ for Case I (both $G$ and $B$ are bounded).}
\label{fig:gc-c1-2}
\end{figure}

The following theorem provides us with a method for computing the unknown real constants $a_1,\ldots,a_{m}$. The theorem will be proved using an approach similar to the approach used in proving Theorems~4.2 and~4.3 in~\cite{NLS16},

\begin{thm}\label{thm:c1}
For each $k=1,2,\ldots,m$, let the function $\gamma_k$ be defined by~\eqref{eq:gam-k}, $\mu_k$ be the unique solution of the integral equation~\eqref{eq:ie-k}, and the piecewise constant function $h_k=(h_{1,k},h_{2,k},\ldots,h_{m+\ell,k})$ be given by~\eqref{eq:h-k}. Then, the boundary values of the function $f$ in~\eqref{eq:rhp-f} are given by
\begin{equation}\label{eq:Af-casI}
A(t)f(\eta(t))=\sum_{k=1}^{m}a_k[\gamma_k(t)+h_k(t)+\i\mu_k(t)]
\end{equation}
and the $m+\ell+1$ unknown real constants $a_1,\ldots,a_{m},c,\nu_1,\ldots,\nu_\ell$ are the components of the unique solution vector of the linear system
\begin{equation}\label{eq:sys-c1}
\left[\begin{array}{ccc:c:ccc}
h_{1,1}      &\cdots &h_{1,m}      &1      &          &         & \\
\vdots       &\ddots &\vdots       &\vdots &          &\bigzero &  \\
h_{m,1}      &\cdots &h_{m,m}      &1      &          &         &  \\ \cdashline{4-7}
h_{m+1,1}    &\cdots &h_{m+1,m}    &0      &-1        &         &\bigzeroU  \\
\vdots       &\ddots &\vdots       &\vdots &          &\ddots   &  \\
h_{m+\ell,1} &\cdots &h_{m+\ell,m} &0      &\bigzeroL &         &-1 \\ \hdashline
1            &\cdots &1            &0      &0         &\cdots   &0 \\
\end{array}\right]
\left[\begin{array}{c}
a_1    \\ \vdots \\ a_{m} \\ \hdashline c \\ \hdashline \nu_1 \\ \vdots \\ \nu_\ell
\end{array}\right]
= \left[\begin{array}{c}
\delta_1 \\  \vdots \\ \delta_m \\  0  \\ \vdots \\ 0 \\ \hdashline 0
\end{array}\right].
\end{equation}
\end{thm}
\begin{proof}
Suppose that $f$ is the analytic function in $G$ with $f(\infty)=0$ for unbounded $G$ and satisfies the boundary condition~\eqref{eq:rhp-f}. Suppose also that $\hat f$ is defined in $G$ by
\begin{equation}\label{eq:hf1-c1}
\hat f(z)=\sum_{k=1}^m a_k f_k(z)
\end{equation}
where $f_k$ are as in Theorem~\ref{thm:ie} and the constants $a_1,\ldots,a_m$ satisfy the condition~\eqref{eq:sum-ak}. Then $\hat f$ is analytic in $G$ with $f(\infty)=0$ for unbounded $G$ and the boundary values of $\hat f$ satisfy
\begin{equation}\label{eq:hf-c1}
\Re\left[A(t)\hat f(\eta(t))\right]=\sum_{k=1}^ma_k\gamma_k(t)+\sum_{k=1}^ma_kh_k(t).
\end{equation}
Then the function $\Psi$ defined by $\Psi(z)=\hat f(z)-f(z)$ is analytic in $G$ with $\Psi(\infty)=0$ for unbounded $G$. Since $m'=m$, it follows from~\eqref{eq:rhp-f} and~\eqref{eq:hf-c1} that
\begin{equation}\label{eq:Psi-bc}
\Re[A(t)\Psi(\eta(t))]=\sum_{k=1}^ma_kh_k(t)+c\cos\theta(t)-\delta(t)-\nu(t).
\end{equation}
The right-hand side is a piecewise constant function, and then Lemma~\ref{lem:solv} implies that $\Psi$ is the zero function and hence $f(z)=\hat f(z)$. Thus, \eqref{eq:Af-casI} follows from~\eqref{eq:fet-k} and~\eqref{eq:hf1-c1}. Further, since $\Psi$ is the zero function, the right-hand side of~\eqref{eq:Psi-bc} is also the zero function and hence
\begin{equation}\label{eq:cd1-c1}
\sum_{k=1}^ma_kh_k+c\cos\theta(t)-\nu(t)=\delta(t).
\end{equation}
Since, in view of~\eqref{eq:theta}, $\cos\theta(t)=1$ for $t\in J_k$ for $k=1,2,\ldots,m$ and $\cos\theta(t)=0$ for $t\in J_k$ for $k=m+1,m+2,\ldots,m+\ell$, then~\eqref{eq:cd1-c1} and~\eqref{eq:sum-ak} imply that the real constants $a_1,\ldots,a_{m},c,\nu_1,\ldots,\nu_\ell$ are the components of a solution vector of the linear system~\eqref{eq:sys-c1}.

To show that the linear system~\eqref{eq:sys-c1} has a unique solution, let $[a_1,\ldots,a_{m},c,\nu_1,\ldots,\nu_\ell]^T$ be a solution to the homogeneous linear system obtained by assuming that the right-hand side of~\eqref{eq:sys-c1} is the zero vector. Then, the homogeneous system implies that
\begin{equation}\label{eq:cd2-c1}
\sum_{k=1}^ma_kh_k+c\cos\theta(t)-\nu(t)=0,  \quad \sum_{k=1}^ma_k=0.
\end{equation}
Assume that the functions $f_k$ are as in Theorem~\ref{thm:ie} and $\hat f$ is defined by~\eqref{eq:hf1-c1}. Hence, in view of~\eqref{eq:hf-c1}, the boundary values of the function $\hat f$ satisfy
\begin{equation}\label{eq:hf2-c1}
\Re\left[A(t)\hat f(\eta(t))\right]=\sum_{k=1}^ma_k\gamma_k(t)+\nu(t)-c\cos\theta(t).
\end{equation}
Then, we define a function $\hat F$ in $G$ by
\begin{equation}\label{eq:hF-c1}
\hat F(z) = \begin{cases}
\displaystyle
(z-\alpha)\hat f(z)-\sum_{k=1}^m a_k \log(z-\alpha_k), & {\rm if\;} G \;{\rm is\; bounded}, \\
\displaystyle
\hat f(z)-\sum_{k=1}^m a_k \log(z-\alpha_k), & {\rm if\;} G \;{\rm is\; unbounded},
\end{cases}
\end{equation}
For unbounded $G$, the function $\hat F(z)$ can be written as
\[
\hat F(z)
=\hat f(z)-\sum_{k=1}^m a_k[\log z+\log(1-\alpha_k/z)]
=\hat f(z)-\log z\sum_{k=1}^m a_k-\sum_{k=1}^m a_k\log(1-\alpha_k/z).
\]
Since $\hat f(\infty)=0$ and $\sum_{k=1}^ma_k=0$, we have $\hat F(\infty)=0$. Thus, the function $\hat F(z)$ is analytic in $G$ for both cases of bounded and unbounded $G$ but it is not necessarily single valued.
In view of~\eqref{eq:A}, the boundary values of the function $\hat F$ satisfy
\[
\Re\left[e^{-\i\theta(t)}\hat F(\eta(t))\right]=\Re\left[A(t)\hat f(\eta(t))\right]-\sum_{k=1}^ma_k\Re\left[e^{-\i\theta(t)}\log(\eta(t)-\alpha_k)\right].
\]
Then by~\eqref{eq:sum-gam-k} and~\eqref{eq:hf2-c1}, we have
\[
\Re\left[e^{-\i\theta(t)}\hat F(\eta(t))\right]=\nu(t)-c\cos\theta(t),
\]
which, in view of~\eqref{eq:theta}, implies that
\begin{subequations}\label{eq:hF-Re-Im}
\begin{equation}\label{eq:hF-Re-Im-1}
\Re\left[\hat F(\eta(t))\right]=-c \quad{\rm for}\quad \eta(t)\in\Gamma_k,\; k=1,2,\ldots,m,
\end{equation}
and
\begin{equation}\label{eq:hF-Re-Im-2}
\Im\left[\hat F(\eta(t))\right]=\nu_k \quad{\rm for}\quad \eta(t)\in\Gamma_k,\; k=m+1,m+2,\ldots,m+\ell.
\end{equation}
\end{subequations}
By differentiation both sides of~\eqref{eq:hF-Re-Im-2} with respect to the parameter $t$, we obtain
\begin{equation}\label{eq:hF-Re-Im-3}
\Im\left[\eta'(t)\hat F'(\eta(t))\right]=0 \quad{\rm for}\quad \eta(t)\in\Gamma_k,\; k=m+1,m+2,\ldots,m+\ell.
\end{equation}
Let the real function $u$ be defined for $z\in G\cup\partial G$ by
\[
u(z)=\Re \hat F(z).
\]
Then $u$ is harmonic in $G$. In view of~\eqref{eq:0-par-u}, we have
\begin{equation}\label{eq:0-par-u2}
\frac{\partial u}{\partial\bn}=\Re\left[\frac{-\i\eta'(t)}{|\eta'(t)|}\hat F'(\eta(t))\right]
=\frac{1}{|\eta'(t)|}\Im\left[\eta'(t)\hat F'(\eta(t))\right].
\end{equation}
Thus, by~\eqref{eq:hF-Re-Im-1}, \eqref{eq:hF-Re-Im-3}, and~\eqref{eq:0-par-u2}, the boundary values of $u$ satisfy the mixed-boundary condition
\begin{eqnarray*}
u(\zeta)&=&-c, \quad \zeta\in\Gamma_k, \quad k=1,2,\ldots,m, \\
\frac{\partial u}{\partial\bn}(\zeta)&=&0, \quad \zeta\in \Gamma_k, \quad k=m+1,m+2,\ldots,m+\ell.
\end{eqnarray*}
Since the above mixed boundary value problem has a unique solution, it is clear that the unique solution is the constant function $u(z)=-c$ for all $z\in G\cup\partial G$.
Thus the real part of $\hat F$ is constant for $z\in G$, and hence, by the Cauchy-Riemann equations, $\hat F$ is constant in $G$, say equal to $C$. This implies that $\hat F(z)=0$ for all $z\in G$ when $G$ is unbounded since $\hat F(\infty)=0$.
Then, for all $z\in G$, it follows from~\eqref{eq:hF-c1} that
\[
\sum_{k=1}^m a_k \log(z-\alpha_k) = \begin{cases}
-C+(z-\alpha)\hat f(z), & {\rm if\;} G \;{\rm is\; bounded}, \\
\hat f(z), & {\rm if\;} G \;{\rm is\; unbounded},
\end{cases}
\]
which implies that that $a_1=a_2=\cdots=a_m=0$ since the functions on the right-hand side are single-valued and the function on the left-hand side is multi-valued.
Thus, for bounded $G$, we have $(z-\alpha)\hat f(z)=C$ for all $z\in G$. By substituting $z=\alpha$, we find $C=0$ and hence $\hat F(z)=0$ for all $z\in G\cup\partial G$. 
Thus for both cases of bounded and unbounded $G$, we have $F(z)=0$ for all $z\in G\cup\partial G$. Hence, it follows from~\eqref{eq:hF-Re-Im} that $c=0$ and $\nu_1=\nu_2=\cdots=\nu_\ell=0$. Thus, the homogeneous linear system has only the trivial solution $[a_1,\ldots,a_{m},c,\nu_1,\ldots,\nu_\ell]^T={\bf 0}$, and hence the matrix of the linear system~\eqref{eq:sys-c1} is non-singular.
\end{proof}

%----------------------------------------
\subsection{Case II: $m'=m-1$}

For this case, $G$ is a bounded multiply connected domain of connectivity $m+\ell$ with $m\ge2$ and $B$ is an unbounded multiply connected domain of connectivity $\ell'=\ell\ge0$ (where $B=\CC$ for $\ell=0$). Here, the simply connected domains $G_1,\ldots,G_{m-1}$ are bounded, the simply connected domain $G_m$ is unbounded, and $\Gamma_m$ is the external boundary component of $G$ (see Figure~\ref{fig:gc-c2}). Further, $A$ is given by the second formula in~\eqref{eq:A} and the functions $\gamma_k$ for $k=1,2,\ldots,m-1$ are given by the first formula in~\eqref{eq:gam-k}. 
For this case, the values of the unknown real constants  $a_1,\ldots,a_{m-1},c,\nu_1,\ldots,\nu_\ell$  can be computed as in the following theorem. Then $a_m$ is computed through~\eqref{eq:am-b}.

\begin{figure}[ht] %
\centerline{
\scalebox{0.5}{\includegraphics[trim=0 0 0 0,clip]{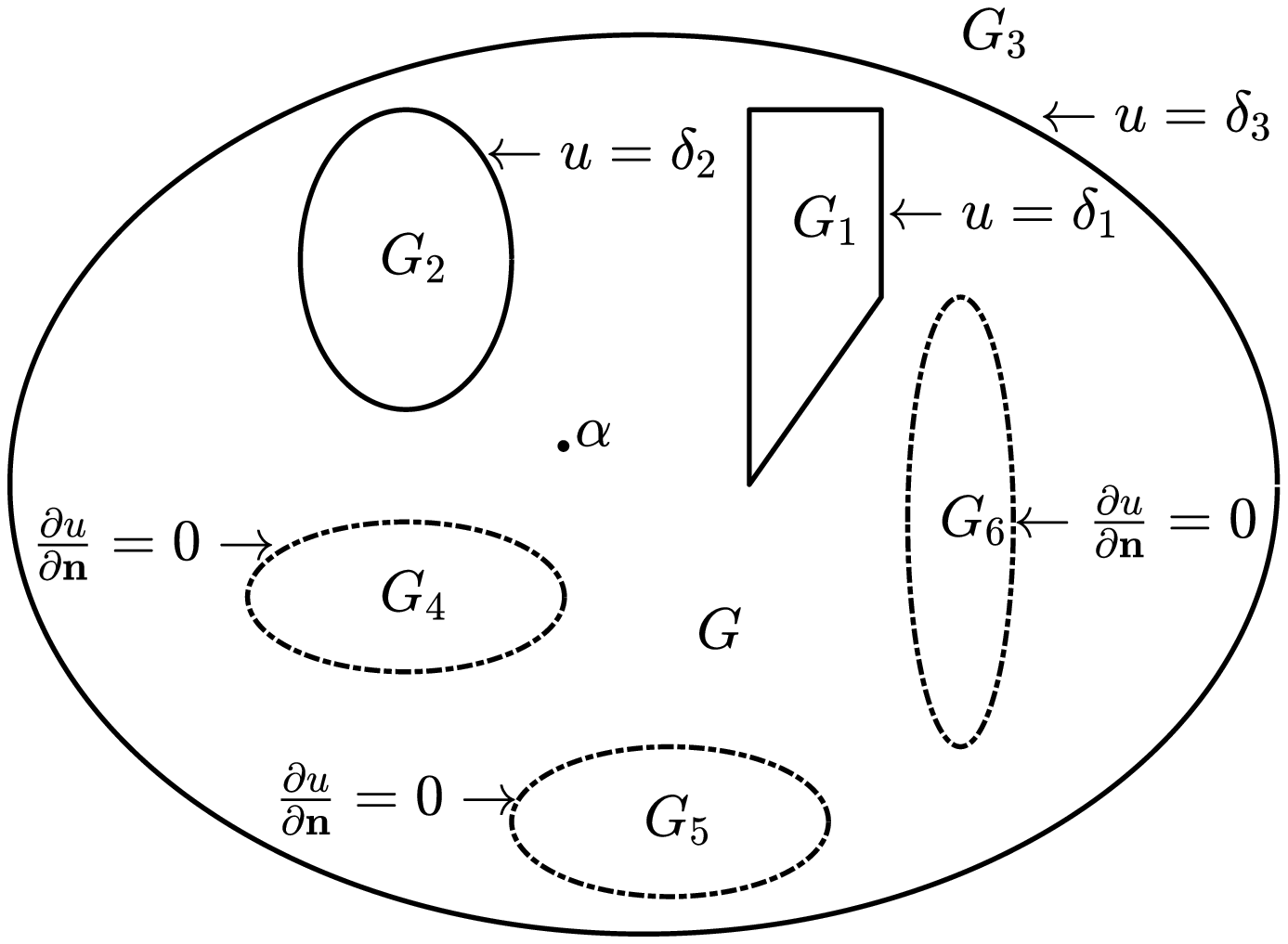}}
}
\caption{An example of a bounded field of the condenser $G$ for $m=3$ and $\ell=3$ for case II ($m'=m-1$, $\ell'=\ell$).}
\label{fig:gc-c2}
\end{figure}

\begin{thm}\label{thm:c2}
For each $k=1,2,\ldots,m-1$, let the function $\gamma_k$ be defined by~\eqref{eq:gam-k}, let $\mu_k$ be the unique solution of the integral equation~\eqref{eq:ie-k}, and let the piecewise constant function $h_k=(h_{1,k},h_{2,k},\ldots,h_{m+\ell,k})$ be given by~\eqref{eq:h-k}.
Then, the boundary values of the function $f$ in~\eqref{eq:rhp-f} are given by
\begin{equation}\label{eq:Af-casII}
A(t)f(\eta(t))=\sum_{k=1}^{m-1}a_k[\gamma_k(t)+h_k(t)+\i\mu_k(t)]
\end{equation}
and the $m+\ell$ unknown real constants $a_1,\ldots,a_{m-1},c,\nu_1,\ldots,\nu_\ell$ are the unique solution of the linear system
\begin{equation}\label{eq:sys-c2}
\left[\begin{array}{ccc:c:ccc}
h_{1,1}      &\cdots &h_{1,m-1}      &1      &          &         & \\
\vdots       &\ddots &\vdots       &\vdots &          &\bigzero &  \\
h_{m,1}      &\cdots &h_{m,m-1}      &1      &          &         &  \\ \cdashline{4-7}
h_{m+1,1}    &\cdots &h_{m+1,m-1}    &0      &-1        &         &\bigzeroU  \\
\vdots       &\ddots &\vdots       &\vdots &          &\ddots   &  \\
h_{m+\ell,1} &\cdots &h_{m+\ell,m-1} &0      &\bigzeroL &         &-1 \\
\end{array}\right]
\left[\begin{array}{c}
a_1    \\ \vdots \\ a_{m-1} \\ \hdashline c \\ \hdashline \nu_1 \\ \vdots \\ \nu_\ell
\end{array}\right]
= \left[\begin{array}{c}
\delta_1 \\  \vdots \\ \delta_m \\  0  \\ \vdots \\ 0
\end{array}\right].
\end{equation}
\end{thm}
\begin{proof}
The theorem can be proved by the same argument as in the proof of Theorem~\ref{thm:c1}.
\end{proof}

%----------------------------------------
\subsection{Computing the capacity $\capa(C)$ and the potential function $u$}

By solving the integral equations~\eqref{eq:ie-k} and then solving the linear system~\eqref{eq:sys-c1} (or \eqref{eq:sys-c2}), we obtain the real constants $a_1,\ldots,a_{m}$. Then, we can compute the capacity $\capa(C)$ from~\eqref{eq:cap-k}. We can also compute the boundary values of the auxiliary analytic function $f(z)$ through~\eqref{eq:Af-casI} or~\eqref{eq:Af-casII}. Then the values of $f(z)$ at interior points $z\in G$ can be computed by Cauchy's integral formula. Since $u(z)=\Re[F(z)]$, it follows from~\eqref{eq:F-u2} and~\eqref{eq:f-g} that the function $u(z)$ is given for $z\in G$ by
\begin{equation}\label{eq:u(z)}
u(z) = \begin{cases}
\displaystyle
c+\Re[(z-\alpha)f(z)]-\sum_{k=1}^{m'} a_k \log|z-\alpha_k|, & {\rm if\;} G \;{\rm is\; bounded}, \\
\displaystyle
c+\Re[f(z)]-\sum_{k=1}^{m'} a_k \log|z-\alpha_k|, & {\rm if\;} G \;{\rm is\; unbounded},
\end{cases}
\end{equation}

%----------------------------------------
\subsection{Outline of the algorithm}
\label{sc:algorithm}

The method presented in this section for computing the capacity $\capa(C)$ and the potential function $u$ can be summarized in the following algorithm. Steps 10--12 are needed only if it is required to compute the values of the potential function.

\begin{algorithm}
\label{alg}
{\rm(Computing the capacity $\capa(C)$ and the potential function $u$).}

\begin{itemize}
	\item[1.] Parametrize the boundary components $\Gamma_j$ by $\eta_j(t)$, $t\in[0,2\pi]$, for $j=1,2,\ldots,m+\ell$, where $\Gamma_j$ for $j=1,2,\ldots,m$ are the boundaries of the plates $E_j$ of the condenser and $\Gamma_j$ for $j=m+1,m+2,\ldots,m$ are the boundary components of the domain $B$.
	\item[2.] If $G$ is bounded and $B$ is unbounded, then we define $m'=m-1$ and $\ell'=\ell$. For this case, the plates $E_1,\ldots,E_{m-1}$ are bounded, the plate $E_m$ is unbounded and $\Gamma_m$ is the external boundary component of $G$.
	\item[3.] If both domains $B$ and $G$ are bounded, then we define $m'=m$ and $\ell'=\ell-1$. For this case, the plates $E_1,\ldots,E_{m}$ are bounded and $\Gamma_{m+\ell}$ is the external boundary component of $G$.
	\item[4.] If both domains $B$ and $G$ are unbounded, then we define $m'=m$ and $\ell'=\ell$. For this case, the plates $E_1,\ldots,E_{m}$ are bounded.
	\item[5.] Define the functions $A$ by~\eqref{eq:A}.
	\item[6.] Define the functions $\gamma_k$ for $k=1,2,\ldots,m'$ by~\eqref{eq:gam-k}.
	\item[7.] For $k=1,2,\ldots,m'$, compute the function $\mu_k$ by solving the integral equation~\eqref{eq:ie-k} and compute the function $h_k$ through~\eqref{eq:h-k}.
	\item[8.] Compute the $m+\ell+1$ real constants $a_1,\ldots,a_{m},c,\nu_1,\ldots,\nu_\ell$ by solving one of the linear system~\eqref{eq:sys-c1} or~\eqref{eq:sys-c2}. For $m'=m-1$, $a_m$ is computed through~\eqref{eq:am-b}.
	\item[9.] Compute the capacity $\capa(C)$ from~\eqref{eq:cap-k}.
	\item[10.] Compute the boundary values of the analytic function $f$ through~\eqref{eq:Af-casI} or~\eqref{eq:Af-casII}.
	\item[11.] Compute the values of $f(z)$ for $z\in G$ by the Cauchy integral formula.
	\item[12.] Compute the values of the potential function $u$ by~\eqref{eq:u(z)}.
\end{itemize}
\end{algorithm}

% FILE:  gcsec502.tex
%%%%%%%%%%%%%%%%%%%%%%%%%%%%%%%%%%%%
%%%%%%%%%%%%%%%%%%%%%%%%%%%%%%%%%%%%
\section{Numerical implementation of the algorithm}
\label{sc:num}

The main steps in the Algorithm~\ref{alg} are steps 7 and 8. In step 8, the size of the linear system is usually quite small and hence we solve it using MATLAB ``backslash'' operator. For step 7, the $m'$ integral equation~\eqref{eq:ie-k} are solved using the MATLAB function \verb|fbie| from~\cite{Nas-ETNA}. 
In the function \verb|fbie|, the integral equations~(\ref{eq:ie-k}) is discretized by the Nystr\"om method with the trapezoidal rule~\cite{Atk97,Tre-Trap}. The size of the obtained linear system is usually large. So, in the function \verb|fbie|, the linear system is solved iteratively using the MATLAB function $\mathtt{gmres}$. The matrix-vector multiplication in $\mathtt{gmres}$ is computed in a fast and efficiently way using the MATLAB function $\mathtt{zfmm2dpart}$ from the toolbox $\mathtt{FMMLIB2D}$~\cite{Gre-Gim12}. The function \verb|fbie| computes also the $m'$ piecewise constant functions $h_k$ in~\eqref{eq:h-k}.

For domains with smooth boundaries, we use the trapezoidal rule with equidistant nodes. 
We discretize each interval $J_k=[0,2\pi]$, for $k=1,2,\ldots,m+\ell$, by $n$ equidistant nodes $s_1, \ldots, s_n$ where
\begin{equation}\label{eq:s_i}
s_k = (k-1) \frac{2 \pi}{n}, \quad k = 1, \ldots, n,
\end{equation}
and $n$ is an even integer. We write $\bs=[s_1,\ldots, s_n]$. Then, we discretize the parameter domain $J$ by the $m+\ell$ copies of $\bs$, 
\begin{equation}\label{eq:bt}
\bt = [\bs, \bs, \ldots, \bs]^T.
\end{equation}
This leads to the discretizations
\begin{equation}\label{eq:dis_s}
\eta(\bt) = [\eta_1(\bs), \eta_2(\bs), \ldots, \eta_{m+\ell}(\bs)]^T, \quad
\eta'(\bt), \quad A(\bt), \quad
\gamma_k(\bt),\; k=1,2,\ldots,m'.
\end{equation}
In MATLAB, these discretized functions are stored in the vectors \texttt{et}, \texttt{etp}, \texttt{A}, \texttt{gamk}, respectively. 
Then the discretizations vectors \texttt{muk} and \texttt{hk} of the functions $\mu_k$ and $h_k$ in~\eqref{eq:ie-k} and~\eqref{eq:h-k} are computed by calling
\begin{verbatim}
  [muk,hk] = fbie(et,etp,A,gamk,n,iprec,restart,tol,maxit).
\end{verbatim}
In the numerical experiments in the next sections, we choose 
$\mathtt{iprec}=5$ (the tolerance of the FMM is $0.5\times 10^{-15}$),  
{\tt restart=[\,]} (GMRES is used without restart), 
{\tt tol=1e-14} (the tolerance of the GMRES method is $10^{-14}$),
and {\tt maxit=100} (the maximum number of GMRES iterations is $100$).
The values $h_{j,k}$ are then computed by taking arithmetic means:
\begin{equation*}
h_{j,k} = \frac{1}{n} \sum_{i=1+(j-1)n}^{jn} h_k(t_i), \quad j=1,2,\ldots,m+\ell, \quad k=1,2,\ldots,m'.
\end{equation*}
These values are used to build the linear system~\eqref{eq:sys-c1} or~\eqref{eq:sys-c2}.
Thus, the computational cost of the overall method for computing the capacity $\capa(C)$ is $O(m'(m+\ell)n\ln n)$ operations for step (7) and $O((m+\ell)^3)$ operations for step (8).

For fast and accurate computing of the Cauchy integral formula in step (11), we use the MATLAB function \verb|fcau| from~\cite{Nas-ETNA}. The function \verb|fcau| is based on using the MATLAB function \verb|zfmm2dpart| in~\cite{Gre-Gim12}. Using the function \verb|fcau|, the Cauchy integral formula can be computed at $p$ interior points in $O(p+(m+\ell)n)$ operations.

For domains with corners (excluding cusps), the trapezoidal rule with equidistant nodes yields only poor convergence and hence the trapezoidal rule with a graded mesh will be used~\cite{Kre90}. 
Equivalently, we can remove the discontinuity of the derivatives of the solution of the integral equation at the corner points by choosing an appropriate one-to-one function $\sigma:J\to J$. Then we parametrize the boundary $\Gamma$ by $\eta(t)=\hat\eta(\sigma(t))$ where $\hat\eta$ is any parametrization function of the boundary $\Gamma$ (see~\cite{Kre90,LSN17} for more details, the above function $\sigma$ is denoted by $\delta$ in~\cite{LSN17}).

The proposed method can be implemented in MATLAB as in the following function \verb|capgc.m|.
 %where the inputs \texttt{et} and \texttt{etp} are given as in~\eqref{eq:dis_s}, \texttt{alphav} is the column vector of the $m'$ auxiliary points $\alpha_j$, \texttt{deltav} is the column vector of the $m$ real numbers $\delta_j$, and $\alpha$ is an auxiliary point in $G$ for bounded $G$.

\begin{lstlisting}
function [cap , uz] = capgc(et,etp,alphav,deltav,m,mp,ell,alpha,z)
% Compute the capacity of the generalized condensers (B,E,delta)
% 
% Input:
% 1,2) et, etp: parametrization of the boundary and its first derivative
% 3) alphav=[alphav(1),...,alphav(mp)]: alphav(j) is an auxiliary point
% interior to the boundary component \Gamma_j
% 4) deltav=[deltav(1),...,deltav(m)]: deltav(j) is the value of the
% potential function u on \Gamma_j
% 5) m: the number of the closed sets E_k
% 6) mp: mp=m-1 if \Gamma_m is the external boundary component of G,
%    o.w., mp=m 
% 7) ell: the multiplicity of the domain B (B=C for ell=0)
% 8) alpha: for bounded G, alpha is an auxiliary point in G
%           for unbounded G, alpha=inf
% 9) z: a row vector of points in G (if it is required to compute u(z))
% 
% Output:
% cap (the capacity of the generalized condensers (C,E,delta)).
% uz (the values of the potential function u(z) if z is given).
%
% Computing the constants \h_{j,k} for j=1,2,...,m+ell and k=1,2,...,mp
ellp = ell ; ellp(abs(alpha)<inf & mp==m)=ell-1; 
n=length(et)/(m+ell); tht=zeros(size(et)); tht(m*n+1:end)=pi/2;
if mp==m & ellp==ell
    A=exp(-i.*tht);
else
    A=exp(-i.*tht).*(et-alpha);
end
for k=1:mp
    for j=1:m+ell
        jv = 1+(j-1)*n:j*n;
        if (ellp==ell)
            gamk{k}(jv,1)=real(exp(-i.*tht(jv)).*clog(et(jv)-alphav(k)));
        else
            gamk{k}(jv,1)=real(exp(-i.*tht(jv)).*...
                clog((et(jv)-alphav(k))./(et(jv)-alpha)));
        end
    end    
    [mu{k},h{k}]=fbie(et,etp,A,gamk{k},n,5,[],1e-14,100);
    for j=1:m+ell
        hjk(j,k)=mean(h{k}(1+(j-1)*n:j*n));
    end
end
% Computing the constants a_k  for k=1,2,...,m
mat=hjk; mat(1:m,mp+1)=1; mat(m+1:m+ell,mp+1)=0; 
mat(1:m,mp+2:mp+ell+1)=0; mat(m+1:m+ell,mp+2:mp+ell+1)=-eye(ell); 
rhs(1:m,1)=deltav; rhs(m+1:m+ell,1)=0; 
if mp==m
    mat(m+ell+1,1:m)=1; mat(m+ell+1,m+1:m+ell+1)=0; rhs(m+ell+1,1)=0;
end
x=mat\rhs; a=x(1:mp,1); c=x(mp+1);
if mp==m-1
    a(m,1)=-sum(a);
end
% Computing the capacity
cap  = (2*pi)*sum(deltav(:).*a(:));
% compute the values of the potential function u(z) if z is given
if nargin==9
    fet = zeros(size(et)); uz=zeros(size(z));
    for k=1:mp
        fet = fet+a(k).*(gamk{k}+h{k}+i.*mu{k})./A;
        uz=uz-a(k)*log(abs(z-alphav(k)));
    end
    if abs(alpha)<inf
        fz=fcau(et,etp,fet,z);
        uz=uz+c+real((z-alpha).*fz);
    else
        fz=fcau(et,etp,fet,z,n,0);
        uz=uz+c+real(fz);
    end
end
end
\end{lstlisting}

In this paper, computations were performed in MATLAB R2017a on an ASUS Laptop with Intel(R) Core(TM) i7-8750H CPU @2.20GHz, 2208 Mhz, 6 Core(s), 12 Logical Processor(s), and 16GB RAM. The computation times presented in this paper were measured with the MATLAB tic toc commands. All the computer codes of our computations are available in the internet link \url{https://github.com/mmsnasser/gc}.

% FILE:  gcsec602.tex
%%%%%%%%%%%%%%%%%%%%%%%%%%%%%%%%%%%%
%%%%%%%%%%%%%%%%%%%%%%%%%%%%%%%%%%%%
%%%%%%%%%%%%%%%%%%%%%%%%%%%%%%%%%%
\section{Numerical Examples - Regular Condensers}

In this section, we shall consider several numerical examples of regular condensers.
Some of these examples either have know capacity or have been considered in the literature. So, we can compare the obtained results with the exact capacity or with known capacity computed by other researchers. For such case, we have $\ell=0$ and $\{\delta_k\}_{k=1}^{m}$ containing exactly two different numbers which are $1$ and $0$.

%%%%%%%%%%%%%%%%%%%%%%%%%%%%%%%%%%
\nonsec{\bf Two circles.}\label{ex:2cir}

In this example, we consider the generalized condenser $C=(B,E,\delta)$ with $B=\CC$ (and hence $\ell=0$), $E=\{E_1,E_2\}$ (and hence $m=2$), and $\delta=\{0,1\}$. The plates of the condenser are given by $E_k=\overline{G_k}$, $k=1,2$, where $G_1=\{z:|z|<1\}$ and $G_2=\{z:|z-a|<r\}$ for $r>0$ and a real number $a$ with $a>1+r$. So, for this example, the generalized condenser reduces to a regular condenser, $\ell'=\ell=0$, and $m'=m=2$. Thus, the field of the condenser, $G$ is the doubly connected domain in the exterior of the two circles $\Gamma_1=\{z\,:\,|z|=1\}$ and $\Gamma_2=\{z\,:\,|z-a|=r\}$ (see Figure~\ref{fig:2cir} (left) for $a=2$ and $r=0.5$). The exact value of conformal capacity is given by $\capa(G)=2\pi/\log(1/q)$ where $q$ is obtained by solving the following equation~\cite{vuo}
\[
\frac{(1+q)^2}{q}=\frac{(1+a-r)(a+r-1)}{r}.
\]

We use the method presented in Section~\ref{sc:num} with $n=2^{10}$ to compute approximate values for the capacity for $a=2$ and for several values of $r$ between $0.01$ and $0.99$. The relative errors for the computed values for this case are presented in Figure~\ref{fig:2cir}(right). The level curves of the function $u$ for $a=2$ and $r=0.5$ are shown in Figure~\ref{fig:2cir} (left).

\begin{figure}[ht] %
\centerline{
\scalebox{0.5}[0.5]{\includegraphics[trim=0 0cm 0cm 0cm,clip]{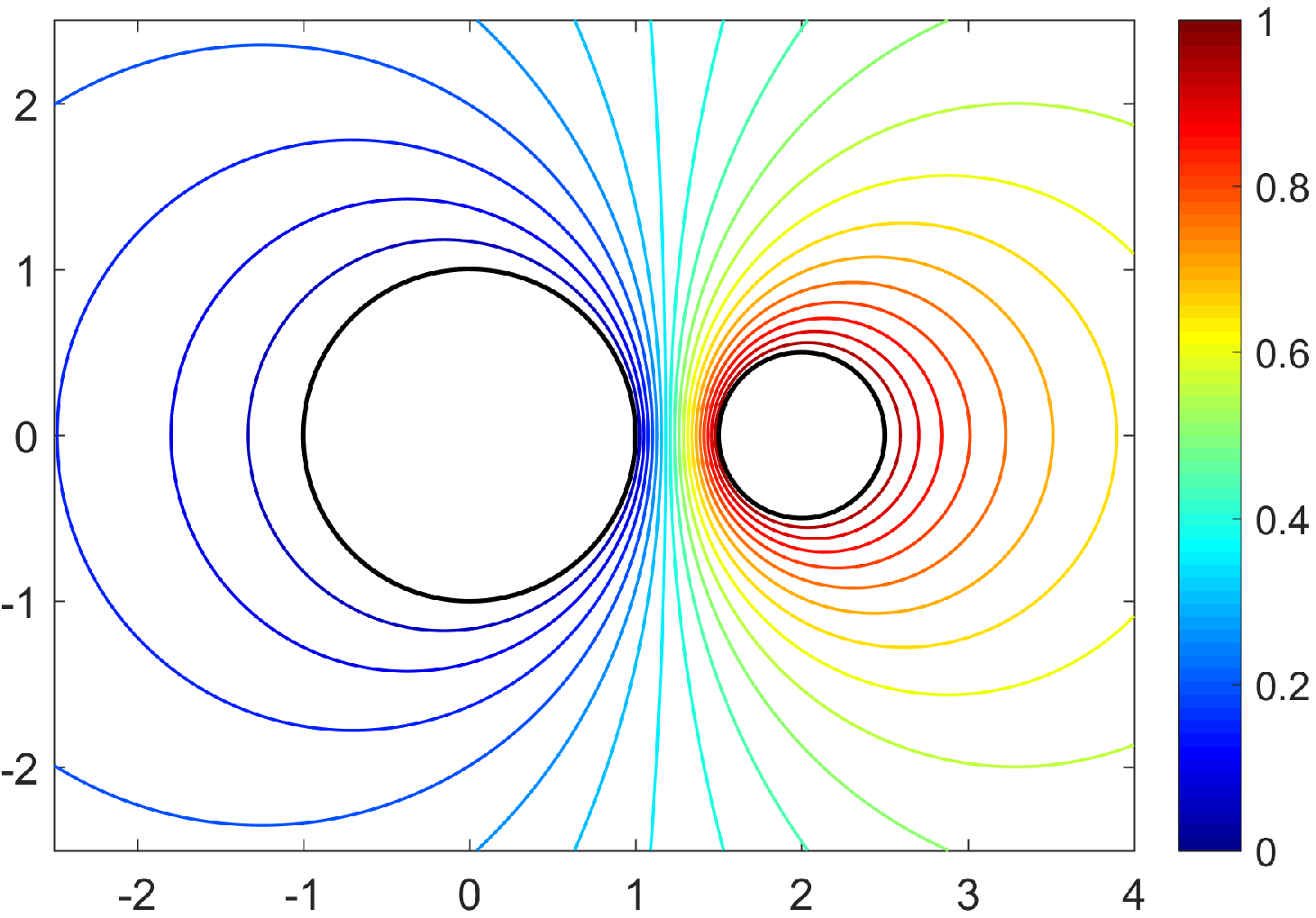}}
\hfill
\scalebox{0.5}[0.5]{\includegraphics[trim=0 0 0 0cm,clip]{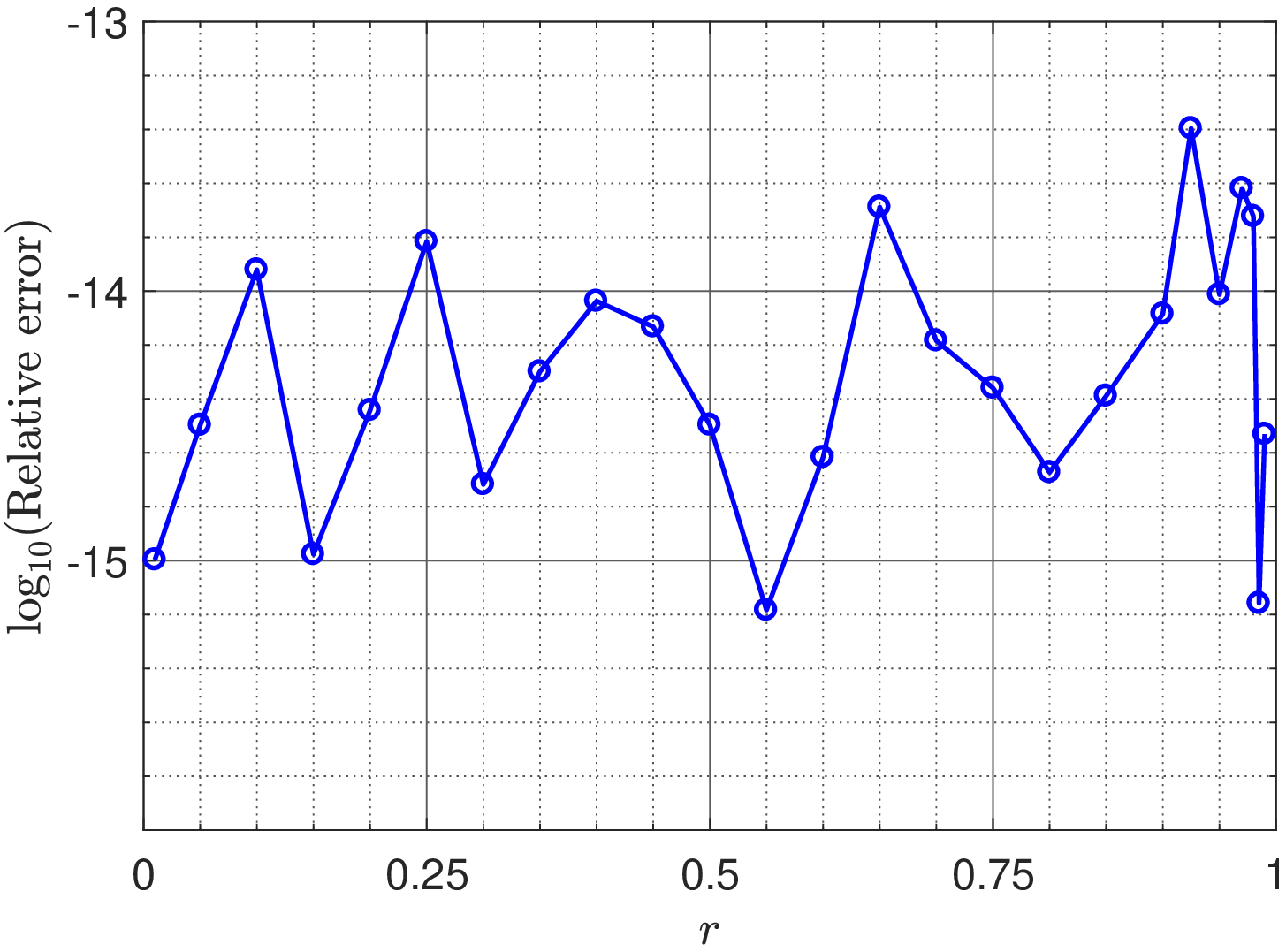}}
}
\caption{The field of the condenser and the level curves of the function $u$ for Example~\ref{ex:2cir} (left) and the relative errors in the computed values (right).}
\label{fig:2cir}
\end{figure}

\nonsec{\bf Square with two triangles.}\label{ex:2tri_squ}

In this example, we consider the generalized condenser $C=(B,E,\delta)$ with $B=\CC$, $E=\{E_1,E_2,E_3\}$ where $E_k=\overline{G_k}$, $k=1,2,3$, and $\delta=\{1,1,0\}$. Here, $G_{1}$ is the interior of the triangles with the vertices $\i a,-(b-a)/\sqrt{3}+\i b,(b-a)/\sqrt{3}+\i b$, $G_{2}$ is the interior of the triangles with the vertices $-\i a,(b-a)/\sqrt{3}-\i b,-(b-a)/\sqrt{3}-\i b$, and $G_3$ is the exterior of the square with the vertices $1+\i,-1+\i,-1-\i,1-\i$. So, $\ell'=\ell=0$, $m=3$, $m'=2$, and the generalized condenser reduces to a regular condenser. The field of the condenser, $G$, is then the bounded multiply connected domain in the exterior of the two triangles and in the interior of the square (see Figure~\ref{fig:2tri_squ}).

This example has been considered in~\cite[Example~7]{bsv} for several values of $a$ and $b$. We use the presented method with $n=3\times2^{13}$ to compute the capacity for the same values of $a$ and $b$ used in~\cite{bsv}. The obtained results as well as the results presented in~\cite{bsv} are shown in Table~\ref{tab:2tri_squ}. The level curves of the function $u$ for $a=0.2$ and $b=0.7$ are shown in Figure~\ref{fig:2tri_squ}.

\begin{figure}[ht] %
\centerline{
\scalebox{0.5}[0.5]{\includegraphics[trim=0 0cm 0cm 0cm,clip]{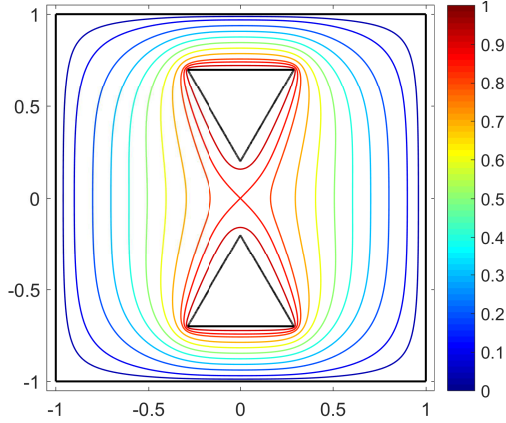}}
}
\caption{The field of the condenser and the level curves of the function $u$ for the condenser in Example~\ref{ex:2tri_squ}.}
\label{fig:2tri_squ}
\end{figure}

\begin{table}[ht]
\caption{The approximate values of the capacity $\capp(C)$ for Example~\ref{ex:2tri_squ}.}
\label{tab:2tri_squ}%
\vskip-0.5cm
\[
\begin{array}{l@{\hspace{1.00cm}}l@{\hspace{1.00cm}}l@{\hspace{1.00cm}}l} \hline %
a         & b     & \mbox{Our\;Method} & \mbox{\cite{bsv}}       \\ \hline %
0.1       & 0.3   & 3.93241437137267   & 3.9324143\\
0.2       & 0.4   & 4.41198623240832   & 4.4119861\\
0.2       & 0.7   & 9.49308124679268   & 9.4930811\\
0.3       & 0.8   & 12.1180118821912   & 12.1180117\\
0.3       & 0.9   & 21.6586490491066   & 21.6586487\\
\hline %
\end{array}
\]
\end{table}

\nonsec{\bf Cantor dust.}\label{ex:cant-d}

Cantor dust is a generalization of the classical Cantor middle third set to dimension two. 
Let $I_0=[0,1]$ and recursively define
\[
I_k = \frac{1}{3}I_{k-1}\cup\left(\frac{1}{3}I_{k-1}+\frac{2}{3}\right),
\quad k \ge 1.
\]
This means that $I_k$ is constructed by ``removing'' the middle one third of 
each interval $I_{k-1}$. For $k=0,1,2,\ldots$, the the closed set $I_k$ consists of $2^k$ closed intervals. Then, we define the closed sets $S_k$ as
\[
S_k=I_k\times I_k, \quad  k \ge 0,
\]
where $S_k$ consists of $4^k$ closed square regions, say $E_1,E_2,\ldots,E_{4^k}$ (see Figure~\ref{fig:cant-d-lc} for $k=1$ (left) and $k=2$ (right)). Then the Cantor dust is defined as 
\[
S = \bigcap_{k=1}^{\infty} S_k.
\]

For $k=0,1,2,\ldots$, we consider the generalized condensers $C_k=(B,E,\delta)$ with $B=\CC$ and $E=\{E_1,E_2,\ldots,E_{4^k}\}$, i.e., we have $4^k$ plates. For the levels of the potential function $\delta=\{\delta_j\}_{j=1}^{4^k}$, we assume $\delta_j=0$ for half of the plates (the plates below the line $y=0.5$) and $\delta_j=1$ for the other half (the plates above the line $y=0.5$). 
Thus, $\ell=0$, $m'=m=4^k$, and the generalized condenser reduces to a regular condenser. The field of the condenser, $G$, is then the unbounded multiply connected domain in the exterior of the closed sets $S_k$ (see Figure~\ref{fig:cant-d-lc}).

The approximate value of the capacity for $k=1,2,3,4,5$ are shown in Table~\ref{tab:cant-d} and the level curves of the function $u$ for $k=1,2$ are shown in Figure~\ref{fig:cant-d-lc}. For each $k$, the method requires solving $m'=4^k$ integral equations. The CPU time presented in Table~\ref{tab:cant-d} shows that the method can be used to compute the capacity $\capp(C_k)$ in reasonable time even when $m'$ becomes large. The presented method is used with $n=2^{9}$. 

\begin{table}[ht]
\caption{The approximate values of the capacity $\capp(C_k)$ for Example~\ref{ex:cant-d}.}
\label{tab:cant-d}%
\vskip-0.5cm
\[
\begin{array}{l@{\hspace{1.00cm}}l@{\hspace{1.00cm}}l@{\hspace{1.00cm}}l} \hline %
k    & m=4^k    & \capa(C_k)       & {\rm Time\;(sec)} \\ \hline %
1    & 4        & 4.652547172280   & 0.96 \\
2    & 16       & 4.562140107251   & 7.33\\
3    & 64       & 4.531267950053   & 87.23\\
4    & 256      & 4.519885740453   & 1312.67\\
5    & 1024     & 4.515629401820   & 19880.56 \\
\hline %
\end{array}
\]
\end{table}

\begin{figure}[ht] %
\centerline{
\scalebox{0.55}{\includegraphics[trim=0 0cm 0cm 0cm,clip]{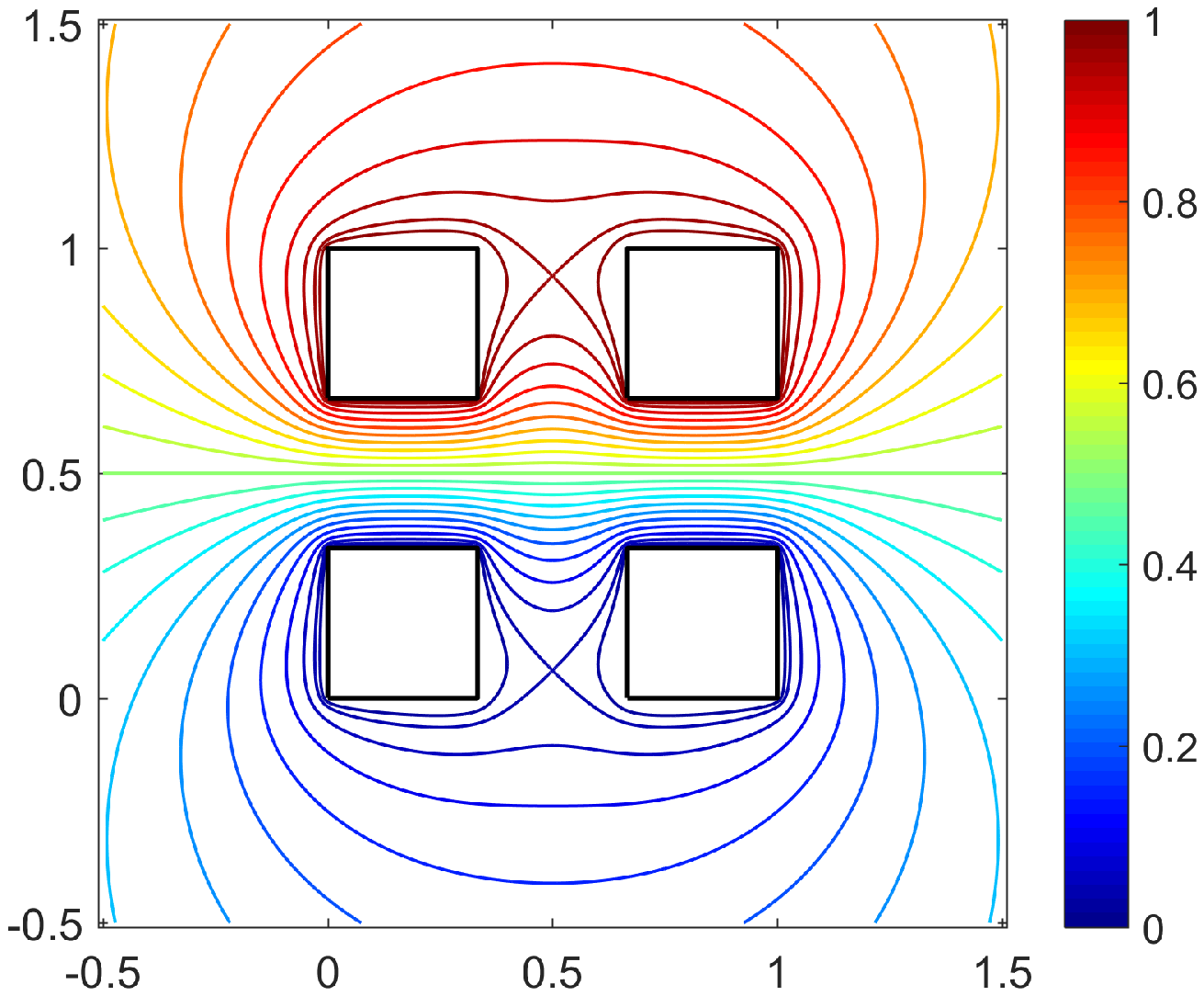}}
\hfill
\scalebox{0.55}{\includegraphics[trim=0 0cm 0cm 0cm,clip]{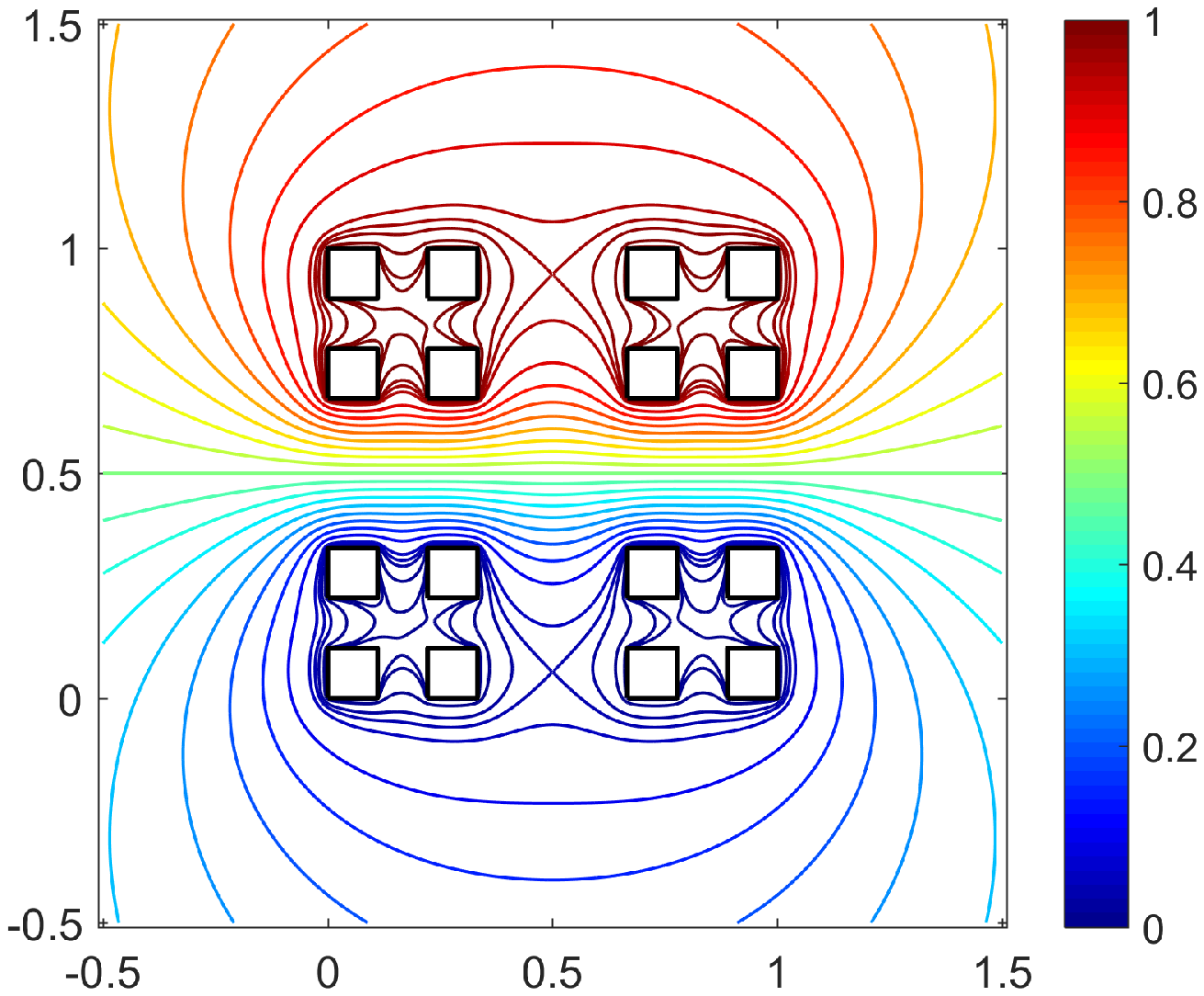}}
}
\caption{The level curves of the function $u$ for the condenser in Example~\ref{ex:cant-d} for $k=1$ (left) and $k=2$ (right).}
\label{fig:cant-d-lc}
\end{figure}

\nonsec{\bf Cantor dust in a circle.}\label{ex:cant-c}

In this example, we consider the generalized condensers $C_k=(B,E,\delta)$ with $B=\CC$, $E=\{E_1,E_2,\ldots,E_{4^k},E_{4^k+1}\}$ where $E_1,E_2,\ldots,E_{4^k}$ are as in Example~\ref{ex:cant-d} and $E_{4^k+1}=\{z\in\CC\,:\,|z-(0.5+0.5\i)|\ge1\}$. For the levels of the potential function, we assume $\delta=\{0,0,\ldots,0,1\}$, i.e., the boundary values of the potential function $u$ are $1$ on the circle $|z-(0.5+0.5\i)|=1$ and $0$ on the boundary of $S_k$, $k=0,1,2,\ldots$. Thus, $\ell'=\ell=0$, $m'=m-1=4^k$, and the generalized condenser reduces to a regular condenser. The field of the condenser, $G$, is then the bounded multiply connected domain in the exterior of the closed sets $S_k$ and in the interior of the circle $|z-(0.5+0.5\i)|=1$ (see Figure~\ref{fig:cant-c-lc}).

The approximate value of the capacity for $k=0,1,\ldots,5$ are shown in Table~\ref{tab:cant-c} and the level curves of the function $u$ for $k=1,2$ are shown in Figure~\ref{fig:cant-c-lc}. As in the previous example, the presented method is used with $n=2^{9}$. 

\begin{table}[ht]
\caption{The approximate values of the capacity $\capp(C_k)$ for Example~\ref{ex:cant-c}.}
\label{tab:cant-c}%
\vskip-0.5cm
\[
\begin{array}{l@{\hspace{1.00cm}}l@{\hspace{1.00cm}}l@{\hspace{1.00cm}}l} \hline %
k    & m'=4^k     & \capp(C_k)        & {\rm Time\;(sec)} \\ \hline %
0    & 1          & 11.953050425798967   & 0.18 \\
1    & 4          & 11.598538784854115   & 1.39 \\
2    & 16         & 11.460679479701366   & 9.49\\
3    & 64         & 11.408998221761493   & 94.22\\
4    & 256        & 11.389646177509054   & 1235.45\\
5    & 1024       & 11.382387009959178   & 19160.17  \\
\hline %
\end{array}
\]
\end{table}

\begin{figure}[ht] %
\centerline{
\scalebox{0.55}{\includegraphics[trim=0 0cm 0cm 0cm,clip]{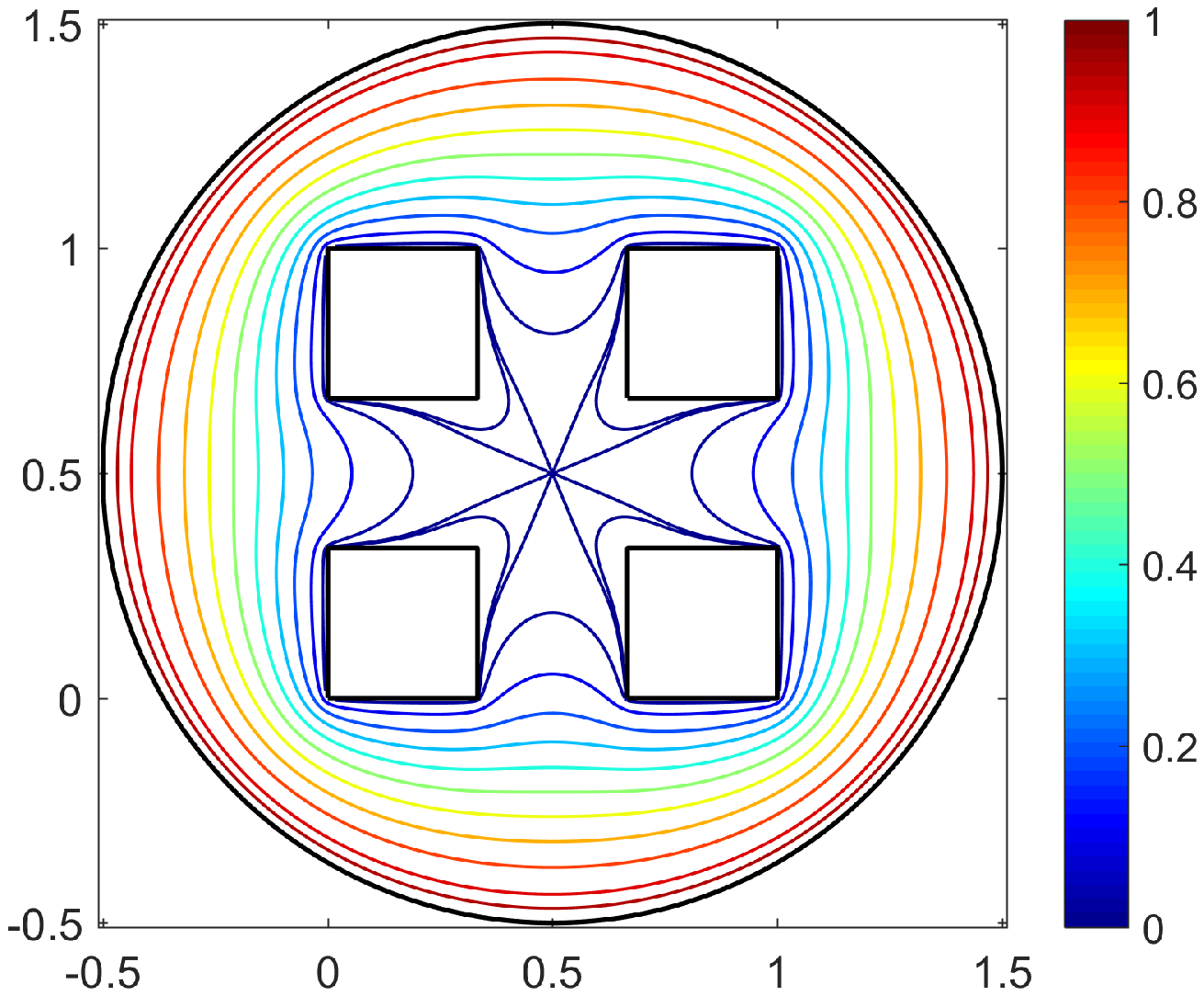}}
\hfill
\scalebox{0.55}{\includegraphics[trim=0 0cm 0cm 0cm,clip]{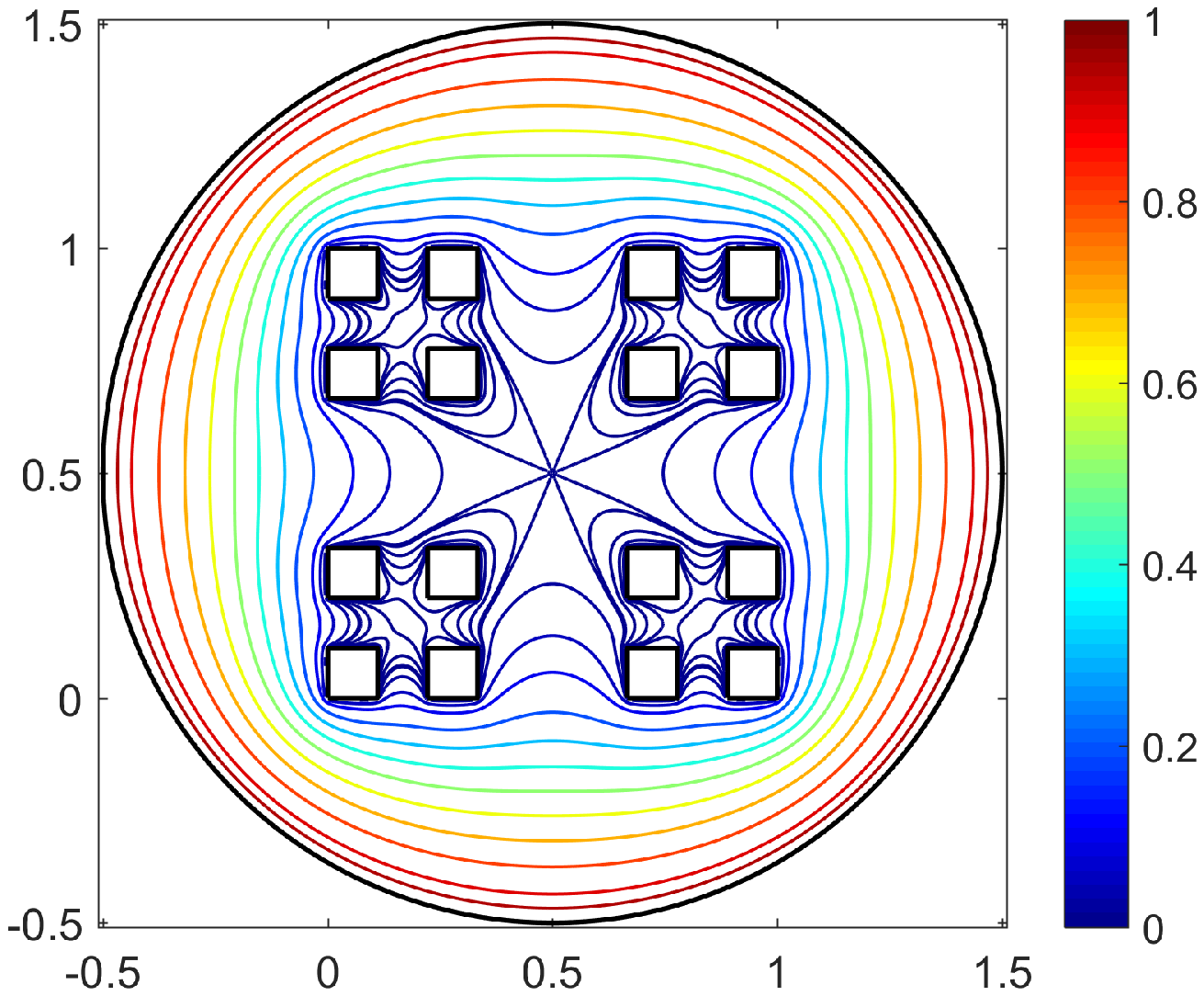}}
}
\caption{The level curves of the function $u$ for the condenser in Example~\ref{ex:cant-c} for $k=1$ (left) and $k=2$ (right).}
\label{fig:cant-c-lc}
\end{figure}

% FILE:  gcsec702.tex
%%%%%%%%%%%%%%%%%%%%%%%%%%%%%%%%%%%%
%%%%%%%%%%%%%%%%%%%%%%%%%%%%%%%%%%%%
%%%%%%%%%%%%%%%%%%%%%%%%%%%%%%%%%%
%%%%%%%%%%%%%%%%%%%%%%%%%%%%%%%%%%
\section{Numerical Examples - Generalized Condensers}

In this section, we shall consider several numerical examples of generalized condensers. For such case, we have either $\ell\ne0$ or $\ell=0$ with $\{\delta_k\}_{k=1}^{m}$ containing at least three different numbers.

%%%%%%%%%%%%%%%%%%%%%%%%%%%%%%%%%%
\nonsec{\bf Six circles.}\label{ex:6cir}

In this example, we assume that $E=\{E_1,E_2\}$ where $E_1$ and $E_2$ are as in Example~\ref{ex:2cir} with $a=2$, i.e., $E_1=\overline{G_1}$ with $G_1=\{z:|z|<1\}$, and $E_2=\overline{G_2}$ with $G_2=\{z:|z-2|<r\}$ where $0<r<1$ (and hence $m'=m=2$). We consider the generalized condenser $C=(B,E,\delta)$ where $\delta=\{0,\delta_2\}$ with a non-zero real number $\delta_2$ for two cases of the domain $B$.

First, we assume that $B$ is the bounded multiply connected domain
\[
B=B_I=\{z:|z|<3,\quad |z+2|>0.9, \quad |z\mp2\i|>0.9\}.
\]
Hence $\ell=4$ and $\ell'=3$. The field of the condenser, $G$, is then the bounded multiply connected domain of connectivity $6$ exterior to the circles $\Gamma_1=\{z\,:\,|z|=1\}$, $\Gamma_2=\{z\,:\,|z-2|=r\}$, $\Gamma_2=\{z\,:\,|z-2|=r\}$, $L_{1,2}=\{z\,:\,|z\mp2i|=0.9\}$, $L_3=\{z\,:\,|z+2|=0.9\}$, and interior to the circles $L_4=\{z\,:\,|z|=3\}$ (see Figure~\ref{fig:6cir} (left) for $r=0.5$).

Second, we assume that $B$ is the unbounded multiply connected domain
\[
B=B_{II}=\{z:|z-6|>3,\quad |z+2|>0.9, \quad |z-(1\pm3\i)|>2\}.
\]
and hence $\ell'=\ell=4$. Thus, $G$ is the unbounded multiply connected domain of connectivity $6$ exterior to the circles $\Gamma_1=\{z\,:\,|z|=1\}$, $\Gamma_2=\{z\,:\,|z-2|=r\}$, $L_{1,2}=\{z\,:\,|z-(1\pm3\i)|=2\}$, $L_3=\{z\,:\,|z+2|=0.9\}$, and $L_4=\{z\,:\,|z-6|=3\}$ (see Figure~\ref{fig:6cir} (center) for $r=0.5$).

As in Example~\ref{ex:2cir}, we use the presented method with $n=2^{10}$. The approximate values of the capacity computed for $r=0.5$ and for several values of $\delta_2$ are presented in Table~\ref{tab:cir}. The level curves of the function $u$ for $r=0.5$ and $\delta_2=1$ are shown in Figure~\ref{fig:6cir} (left, center). Figure~\ref{fig:6cir} (right) shows the approximate values of the capacity computed for $\delta_2=1$ and for several values of $r$ between $0.01$ and $0.99$. We see from Table~\ref{tab:cir} and Figure~\ref{fig:6cir} (right) that the capacity of the condenser $C=(B,E,\delta)$ for $B=B_{I}$ and $B=B_{II}$ is less than the capacity for $B=\CC$ (Example~\ref{ex:2cir}).  

\begin{figure}[ht] %
\centerline{
\scalebox{0.4}{\includegraphics[trim=0 -0cm 0cm 0cm,clip]{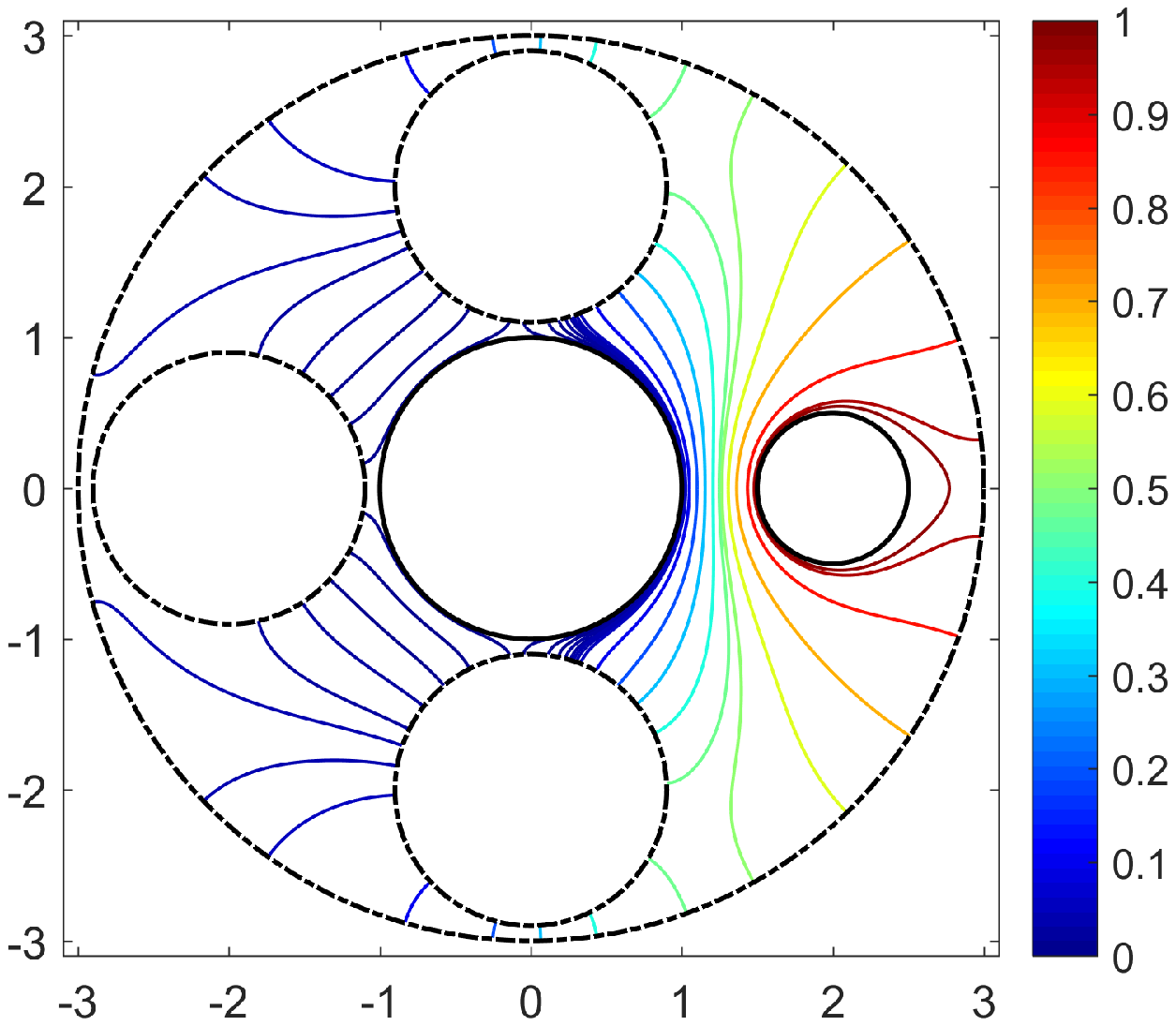}}
\hfill
\scalebox{0.4}{\includegraphics[trim=0 -0cm 0cm 0cm,clip]{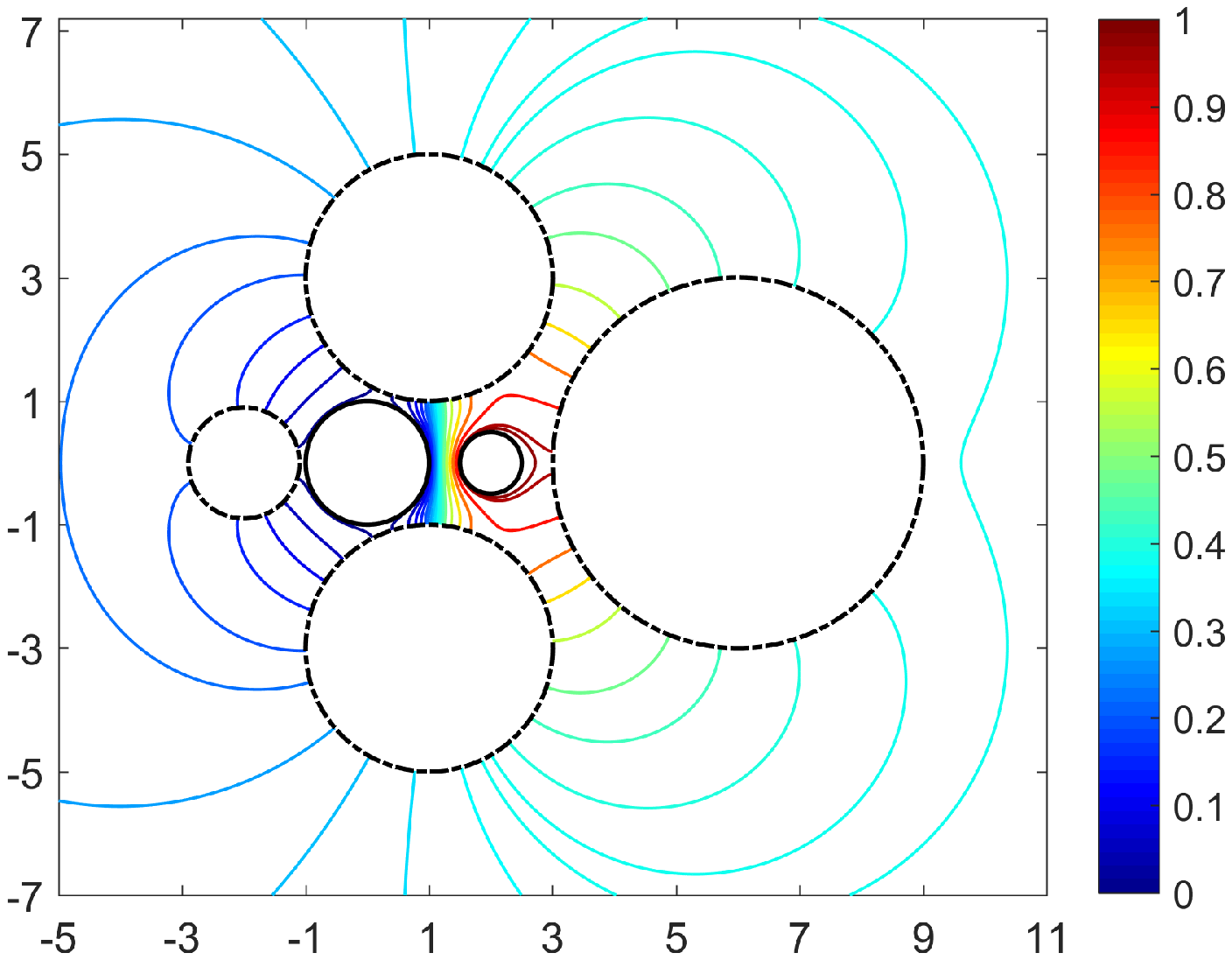}}
\hfill
\scalebox{0.4}{\includegraphics[trim=0 0 0 0cm,clip]{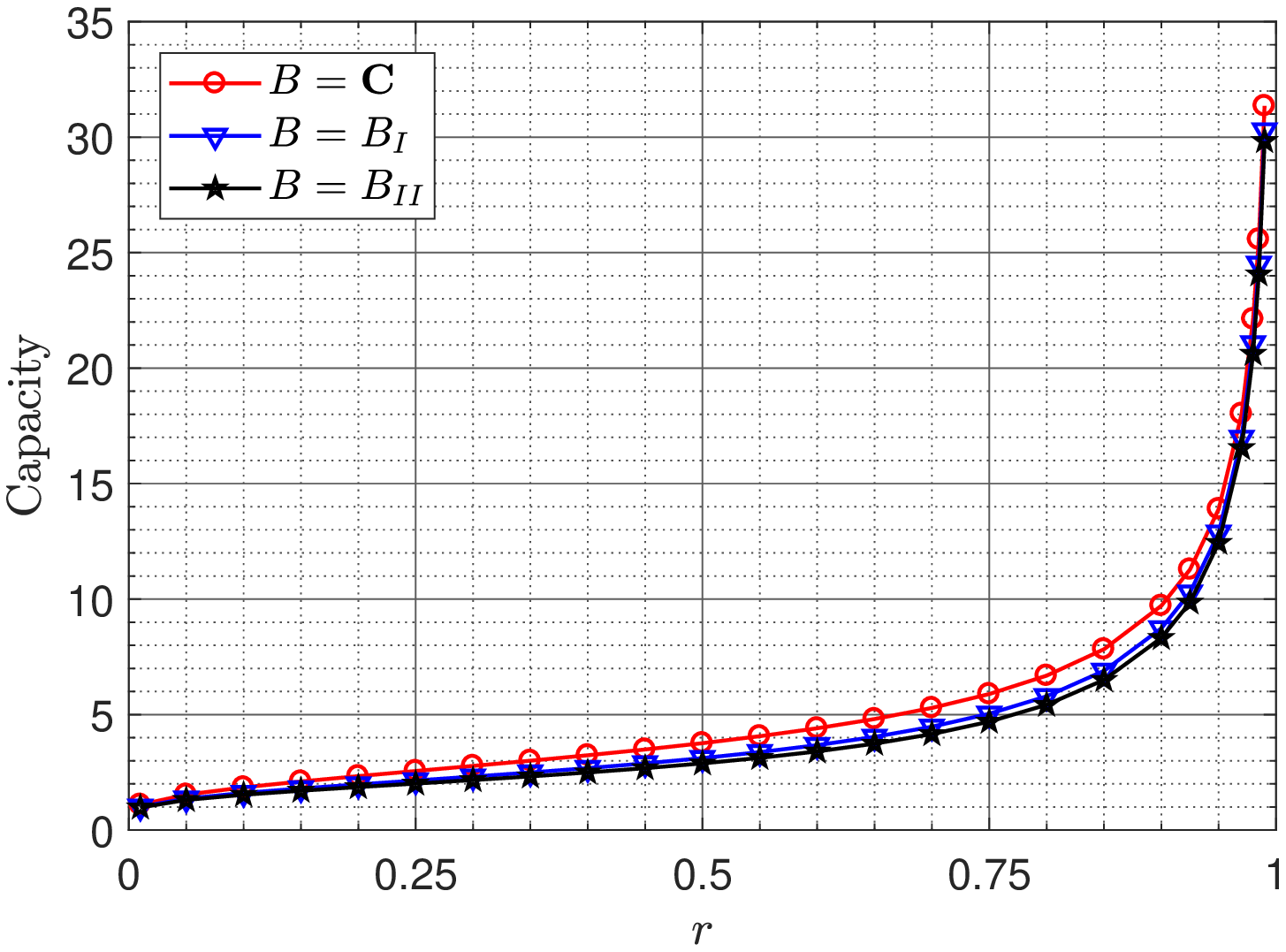}}
}
\caption{The field of the condenser and the level curves of the function $u$ for $B=B_{I}$ (left) and $B=B_{II}$ (center); and the approximate values of the capacity for $\delta_2=1$ (right).}
\label{fig:6cir}
\end{figure}

\begin{table}[ht]
\caption{The approximate values of the capacity $\capp(C)$ for Example~\ref{ex:6cir}.}
\label{tab:cir}%
\vskip-0.5cm
\[
\begin{array}{l@{\hspace{1.00cm}}l@{\hspace{1.00cm}}l@{\hspace{1.00cm}}l} \hline %
\delta_2   & B=B_{I}             & B=B_{II}            & B=\CC            \\ \hline %
0.15       & 0.070116283201223   & 0.064979770350752   & 0.084657798864524\\
0.30       & 0.280465132804894   & 0.259919081403007   & 0.338631195458096\\
0.45       & 0.631046548811011   & 0.584817933156765   & 0.761920189780715\\
0.60       & 1.121860531219576   & 1.039676325612028   & 1.354524781832383\\
0.15       & 1.752907080030588   & 1.624494258768793   & 2.116444971613098\\
0.90       & 2.524186195244046   & 2.339271732627063   & 3.047680759122861\\
\hline %
\end{array}
\]
\end{table}

\nonsec{\bf Five circles.}\label{ex:5cir}

In this example, we consider the generalized condenser $C=(B,E,\delta)$ with $B=\CC$, $E=\{E_1,\ldots,E_5\}$, and $\delta=\{1,2,3,4,0\}$. The plates of the condenser are given by $E_k=\overline{G_k}$, $k=1,\ldots,m$, where $G_{1,3}=\{z:|z\mp2|<1\}$, $G_{2,4}=\{z:|z\mp2\i|<r\}$, and $G_{5}=\{z:|z|>4\}$. So, $\ell'=\ell=0$, $m=5$, and $m'=4$. The field of the condenser, $G$, is then the bounded multiply connected domain in the exterior of the four circles $\Gamma_{1,3}=\{z\,:\,|z\mp2|=1\}$ and $\Gamma_{2,4}=\{z\,:\,|z\mp2\i|=1\}$; and in the interior of the circle $\Gamma_{5}=\{z\,:\,|z|=4\}$ (see Figure~\ref{fig:5cir}).

The approximate values of the capacity obtained with several values of $n$ are shown in Table~\ref{tab:5cir}. Figure~\ref{fig:5cir} shows the level curves of the function $u$ obtained with with $n=2^{10}$.

\begin{table}[ht]
\caption{The approximate values of the capacity $\capp(C)$ for Example~\ref{ex:5cir}.}
\label{tab:5cir}%
\vskip-0.5cm
\[
\begin{array}{l@{\hspace{1.00cm}}l} \hline %
n          & \capp(C)            \\ \hline %
2^5        & 140.5271930046695   \\
2^6        & 140.5271935663499   \\
2^7        & 140.5271935663502   \\
2^8        & 140.5271935663485   \\
2^9        & 140.5271935663483   \\
2^{10}     & 140.5271935663559   \\
\hline %
\end{array}
\]
\end{table}

\begin{figure}[ht] %
\centerline{
\scalebox{0.5}[0.5]{\includegraphics[trim=0 0cm 0cm 0cm,clip]{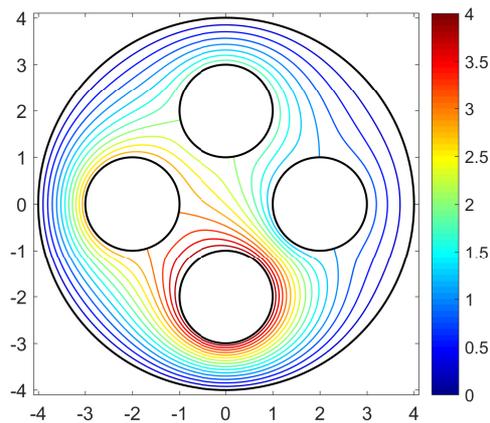}}
}
\caption{The field of the condenser and the level curves of the function $u$ for the condenser in Example~\ref{ex:5cir}.}
\label{fig:5cir}
\end{figure}

\nonsec{\bf Sierpinski carpet.}\label{ex:carpet}

The Sierpinski carpet is another generalization of the Cantor set to dimension two.
The construction of the Sierpinski carpet begins with a square $S_0$. The square $S_0$ is subdivided into $9$ congruent subsquares in a $3$-by-$3$ grid, and the central subsquare is removed to obtain $S_1$. Then, we subdivide each of the $8$ remaining solid squares into $9$ congruent squares and remove the center square from each to obtain $S_2$. The same procedure is then applied recursively to obtain $S_3$, $S_4$, $\ldots$\,, where
\[
S_0\supset S_1\supset S_2\supset S_3\supset S_4\supset \cdots ,
\] 
(see Figure~\ref{fig:carpet-lc1} for $S_2$ (left) and $S_3$ (right)).
Then the Sierpinski carpet is defined as 
\[
S = \bigcap_{k=0}^{\infty} S_k.
\]

For $k=0,1,2,\ldots$, the domain $\hat{S}_k=S_k\backslash\partial S_k$ is a multiply connected domain of connectivity $1+\sum_{j=0}^{k}8^j$. The domain $\hat{S}_k$ has $1+\sum_{j=0}^{k}8^j$ boundary components which all are squares. We will distinguish here two of these squares, namely, the external square which will be called $\Gamma_2$ and internal square which was removed from $S_0$ to obtain $S_1$ and it will be called $\Gamma_1$. The other $-1+\sum_{j=0}^{k}8^j$ squares are in the domain between $\Gamma_1$ and $\Gamma_2$.
Let $B$ be the multiply connected domain obtained by removing these $-1+\sum_{j=0}^{k}8^j$ squares and the domains interior to these squares from the extended complex place $\hat\CC$. Let also $E_{1}=\overline{G}_1$ where $G_1$ is the domain interior to $\Gamma_1$ and $E_{2}=\overline{G}_2$ where $G_2$ is the domain exterior to $\Gamma_2$.
In this example, we consider the generalized condensers $C_k=(B,E,\delta)$ with $E=\{E_1,E_2\}$ and $\delta=\{0,1\}$. Thus, $\ell'=\ell=-1+\sum_{j=0}^{k}8^j$, $m=2$, and $m'=1$. The field of the condenser, $G$, is then the bounded multiply connected domain $\hat S$ (see Figure~\ref{fig:carpet-lc1}).

The approximate value of the capacity for $k=0,1,2,3,4$ are shown in Table~\ref{tab:carpet1} and the level curves of the function $u$ for $k=2,3$ are shown in Figure~\ref{fig:carpet-lc1}. The presented method is used with $n=2^{10}$.
For this example, we have $m'=1$ and hence we need to solve only one integral equation to compute $\capp(C_k)$ for each $k$. The presented method can be used to compute the capacity even when the number of squares is too high. For example, to compute $\capp(C_k)$ for $k=5$, the multiplicity of the domain $G$ is $4682$ and hence, for $n=2^{10}$, the size of the linear system obtained by discretization the integral equation is $4794368$ by $4794368$. Although the size of the system is too high, the presented method requires only $400$ seconds to compute the capacity.

\begin{table}[ht]
\caption{The approximate values of the capacity $\capp(C_k)$ for Example~\ref{ex:carpet}.}
\label{tab:carpet1}%
\vskip-0.5cm
\[
\begin{array}{l@{\hspace{1.00cm}}l@{\hspace{1.00cm}}l@{\hspace{1.00cm}}l} \hline %
k    & m+\ell     & \capp(C_k)         & {\rm CPU\; time\; (sec)}   \\ \hline %
1    & 2          & 6.215546324111108  & 0.25   \\
2    & 10         & 5.088779139415422  & 0.64   \\
3    & 74         & 4.076130615454810  & 3.00   \\
4    & 586        & 3.258035364401146  & 29.69   \\
5    & 4682       & 2.600902059654094  & 399.97   \\
\hline %
\end{array}
\]
\end{table}

\begin{figure}[ht] %
\centerline{
\scalebox{0.55}{\includegraphics[trim=0 0cm 0cm 0cm,clip]{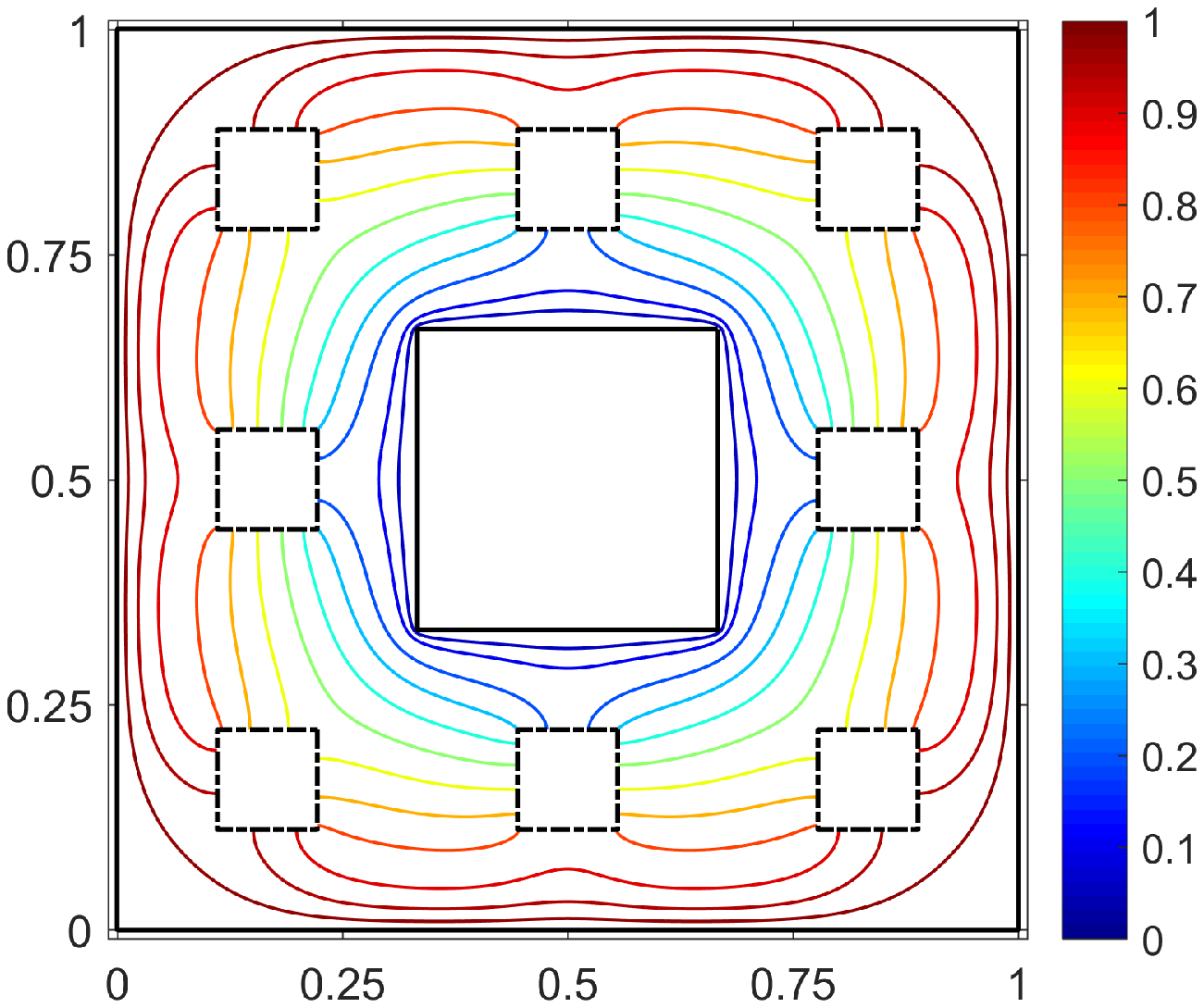}}
\hfill
\scalebox{0.55}{\includegraphics[trim=0 0cm 0cm 0cm,clip]{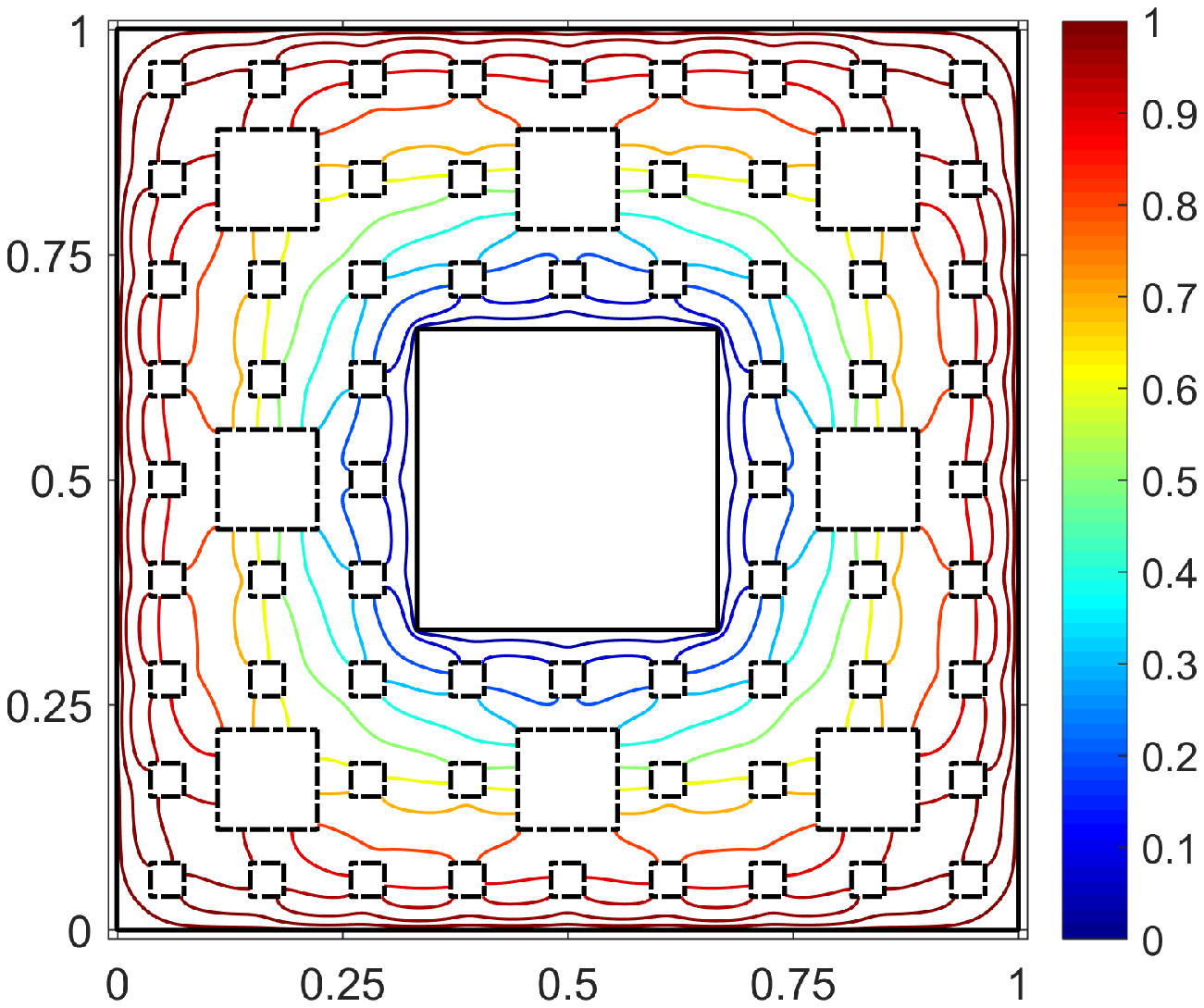}}
}
\caption{The level curves of the function $u$ for the condenser in Example~\ref{ex:carpet} for $k=2$ (left) and $k=3$ (right).}
\label{fig:carpet-lc1}
\end{figure}

% FILE:  gcsec602.tex
%%%%%%%%%%%%%%%%%%%%%%%%%%%%%%%%%%%%
%%%%%%%%%%%%%%%%%%%%%%%%%%%%%%%%%%%%
%%%%%%%%%%%%%%%%%%%%%%%%%%%%%%%%%%
\section{Condensers with slit plates}

The method presented above can be used to compute the capacity of only condensers bordered by smooth or piecewise smooth boundaries. Since the Dirichlet integral is conformally invariant, the capacities for the cases for which the plates of the condenser are rectilinear slits can be computed with the help of conformal mappings as in the following examples.

\nonsec{\bf Three slits: regular condenser.}\label{ex:3slits}

In this example, we consider the generalized condenser $C=(B,E,\delta)$ with $B=\CC$, $E=\{E_1,E_2,E_3\}$ where $E_1=[-c,-1]$, $E_2=[a,b]$, and $E_3=[1,c]$, $-1<a<b<1<c$. For  the levels of the potential of the plates, we consider two cases: $\delta=\{1,1,0\}$  and $\delta=\{0,1,0\}$. So, $\ell=0$, $m=3$, and the generalized condenser reduces to a regular condenser. This example has been considered in~\cite[Example~6]{bsv} for several values of $a$ and $b$.

Here, the field of the generalized condenser, $G$, is the unbounded triply connected domain in the exterior of the three slits  $E_1$, $E_2$, and $E_3$ (see Figure~\ref{fig:3slits}). Hence, the domain $G$ for this generalized condenser is not bordered by Jordan curves. So, the method presented above is not directly applicable to such a domain $G$. Thus, to compute the capacity of this condenser, we first map this domain onto a domain $\hat G$ bordered by smooth Jordan curves so that our method can be used. An iterative numerical method for computing such a domain $\hat G$ has been presented recently in~\cite{NG18}. Using this iterative method, a conformally equivalent domain $\hat G$ bordered by ellipses can be obtained as in Figure~\ref{fig:3slits} (right). For details on the iterative method for computing the domain $G$, we refer the reader to~\cite{NG18}.

\begin{figure}[ht] %
\centerline{
\scalebox{0.5}[0.5]{\includegraphics[trim=0 0cm 0cm 0cm,clip]{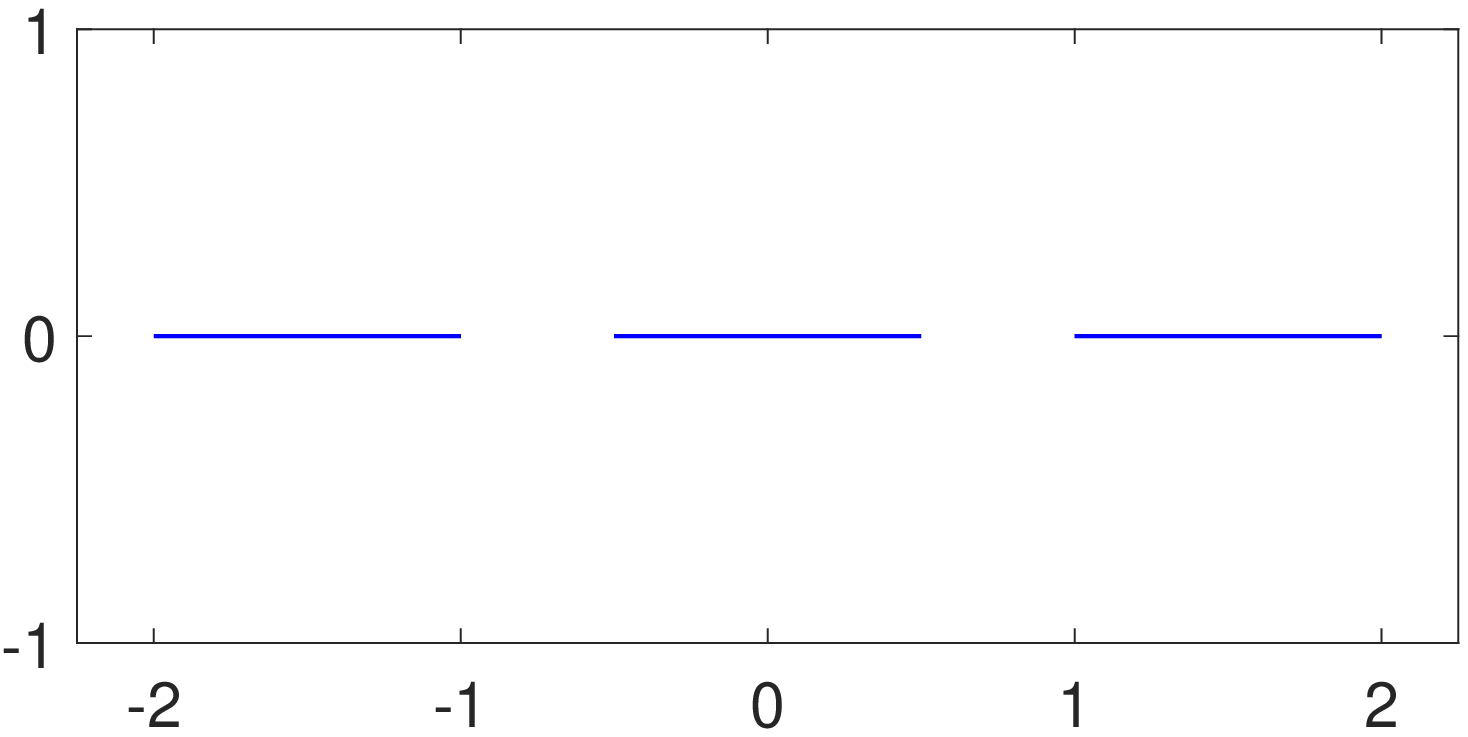}}
\hfill
\scalebox{0.5}[0.5]{\includegraphics[trim=0 0cm 0cm 0cm,clip]{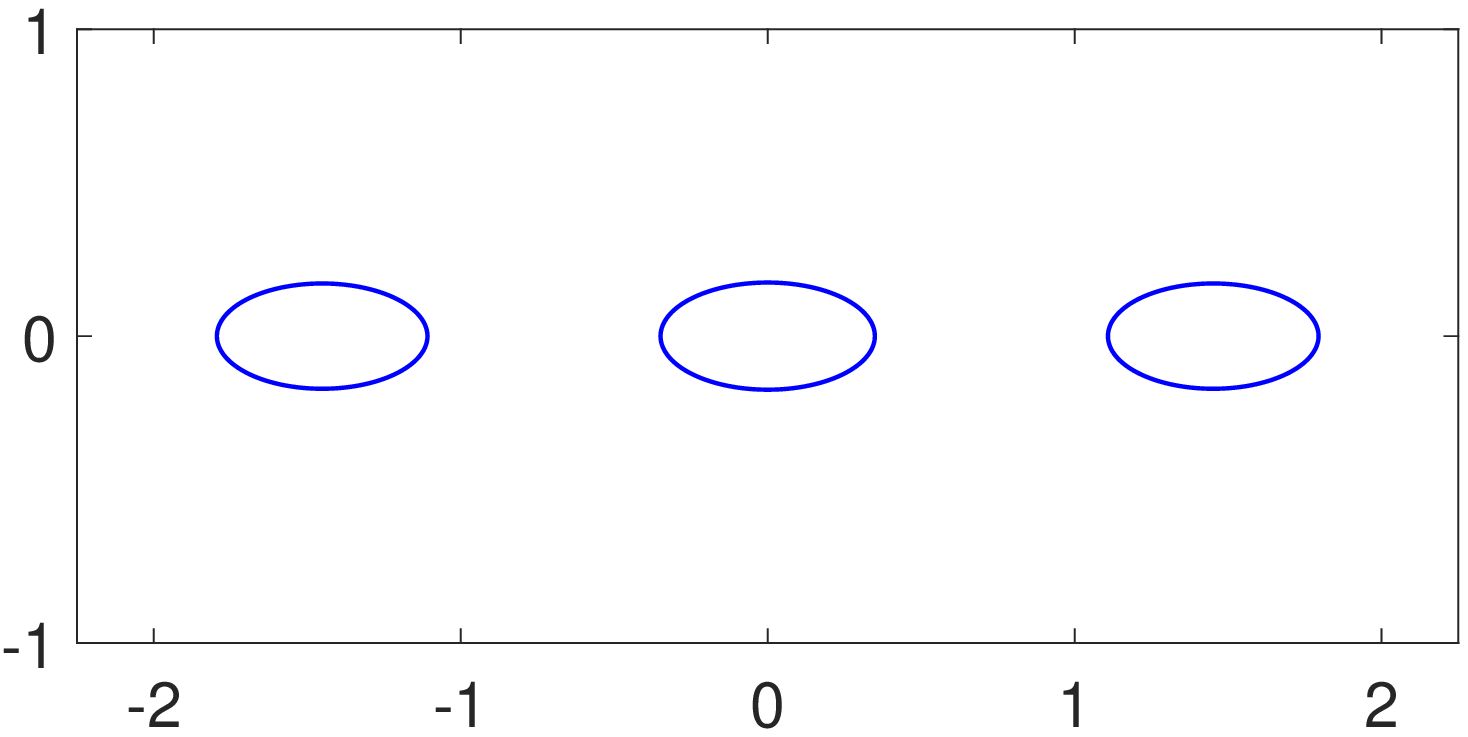}}
}
\caption{The domains $G$ (left) and $\hat G$ (right) for the condenser in Example~\ref{ex:3slits}.}
\label{fig:3slits}
\end{figure}

Since the Dirichlet integral is conformally invariant, the capacity for the new domain $\hat G$ is the same as the capacity for the original domain $G$. For the new domain $\hat G$, we use the presented method with $n=2^{11}$ for several values of the constants $a$, $b$, and $c$ (for the same values used in~\cite{bsv}). The level curves of the function $u$ for $a=-0.5$, $b=0.5$, and $c=2$ are shown in Figure~\ref{fig:3slits-lc}. The obtained approximate values of the capacity as well as the results presented in~\cite{bsv} are shown in Table~\ref{tab:3slits}.

\begin{figure}[ht] %
\centerline{
\scalebox{0.55}{\includegraphics[trim=0 2cm 0cm 2cm,clip]{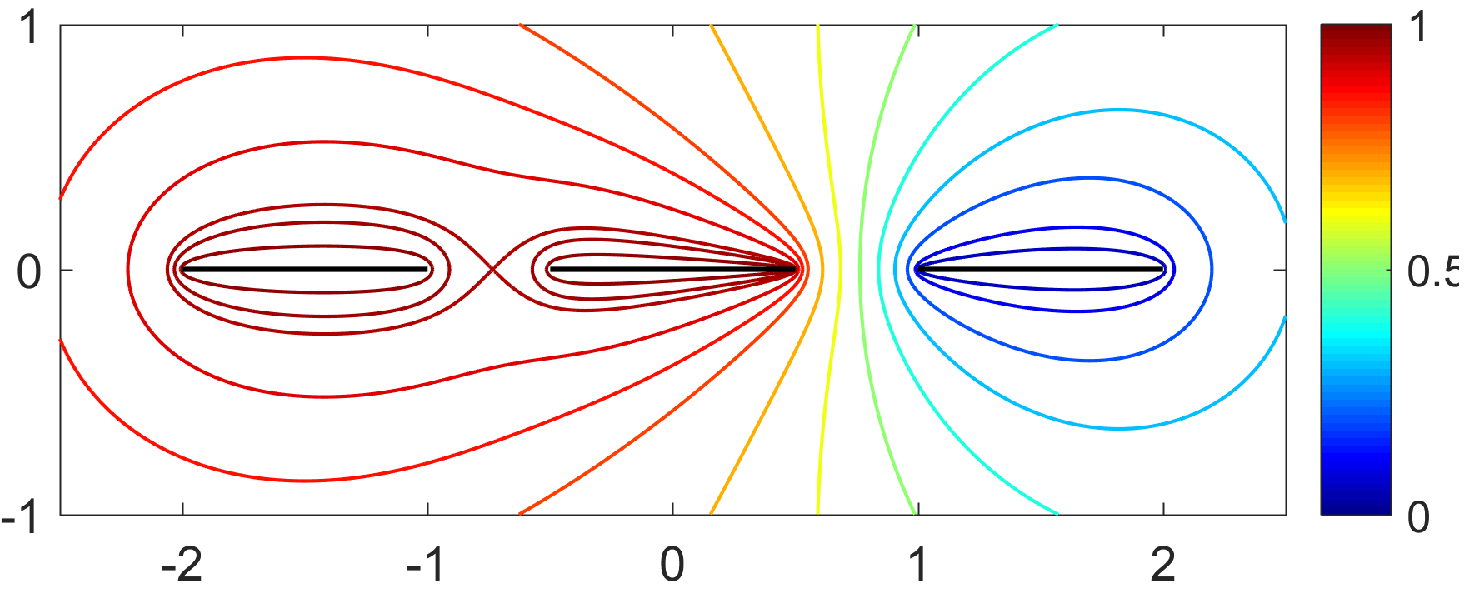}}
\hfill
\scalebox{0.55}{\includegraphics[trim=0 2cm 0cm 2cm,clip]{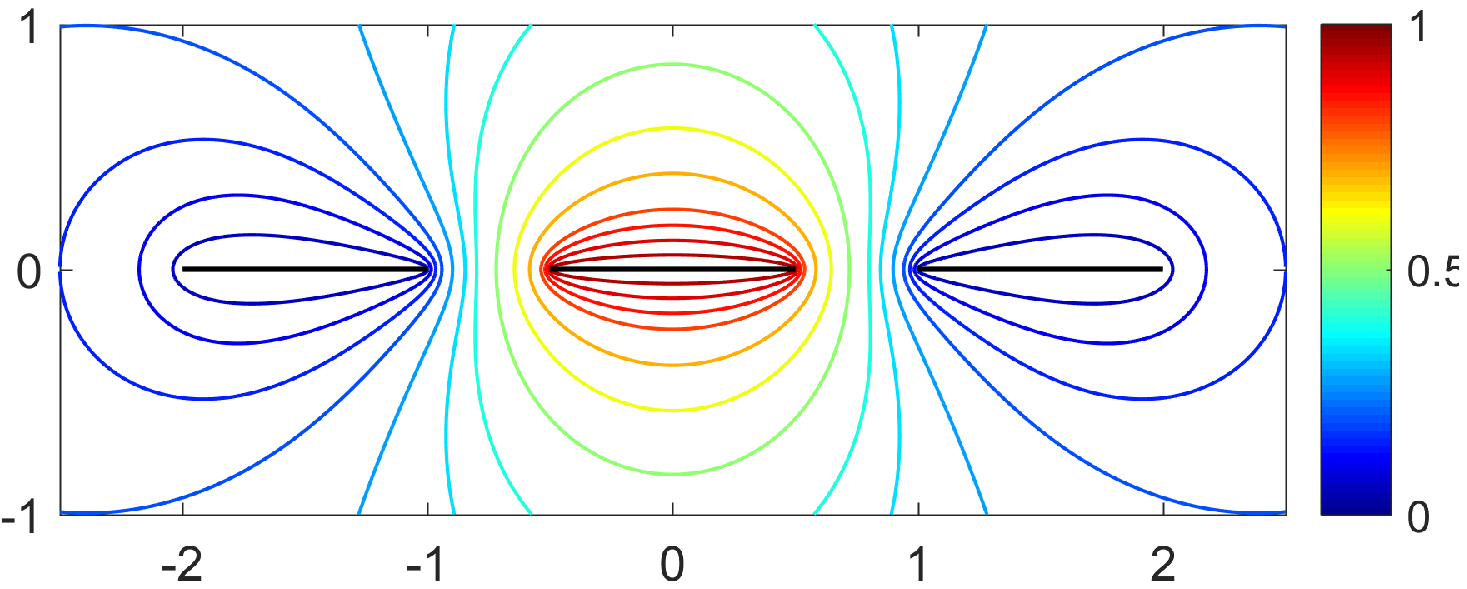}}
}
\caption{The level curves of the function $u$ for the condenser in Example~\ref{ex:3slits} for Case I (left) and Case II (right).}
\label{fig:3slits-lc}
\end{figure}

\begin{table}[ht]
\caption{The approximate values of the capacity $\capp(C)$ for Example~\ref{ex:3slits}.}
\label{tab:3slits}%
\vskip-0.5cm
\[
\begin{array}{r@{\hspace{0.5cm}}l@{\hspace{0.5cm}}l|l@{\hspace{0.5cm}}l||l@{\hspace{0.5cm}}l} \hline %
       &       &     & \multicolumn{2}{c||}{\mbox{Case\;I}} & \multicolumn{2}{c}{\mbox{Case\;II}} \\\cline{4-7}
a      & b     & c   & \mbox{Our\;Method} & \mbox{\cite{bsv}} & \mbox{Our\;Method} & \mbox{\cite{bsv}}       \\ \hline %
-0.9   & 0     & 2   & 1.708669509849820 & 1.7086693 & 3.453772340126319  & 3.4537720\\
-0.5   & 0.5   & 2   & 2.095326566730911 & 2.0953263 & 2.941023714396430  & 2.9410234\\
-0.9   & 0.9   & 2   & 3.067636432954407 & 3.0676361 & 5.187751867577839  & 5.1877511\\
 0     & 0.9   & 2   & 3.033274793073555 & 3.0332745 & 3.453772340126327  & 3.4537719\\
-0.5   & 0.5   & 3   & 2.412575260903909 & 2.4125750 & 3.048687933334055  & 3.0486876\\
-0.7   & 0.2   & 3   & 2.131839309436634 & 2.1318391 & 3.017210220380872  & 3.0172100\\
 0.5   & 0.8   & 3   & 2.807123923176794 & 2.8071236 & 2.312108724455613  & 2.3121085\\
\hline %
\end{array}
\]
\end{table}

\nonsec{\bf Three slits: generalized condenser.}\label{ex:3slits-g}

In this example, we consider the generalized condenser $C=(B,E,\delta)$ with $B=\CC\backslash[a,b]$, $E=\{E_1,E_2\}$ where $E_1=[-c,-1]$, $E_2=[1,c]$, and $\delta=\{0,1\}$, $-1<a<b<1<c$. So, here we have $\ell=1$, $m=2$. The domain $G$ of condenser here is the same as in Example~\ref{ex:3slits} and the Dirichlet boundary condition on the middle slit is replaced with Neumann condition. Thus, as in Example~\ref{ex:3slits}, we compute first a conformally equivalent domain $\hat G$ bordered by smooth Jordan curves. Then, for the domain $\hat G$, we use the presented method with $n=2^{11}$ for the same values of the constants $a$, $b$, and $c$ used in Example~\ref{ex:3slits}. The obtained results are presented in Figure~\ref{fig:3slits-lc-g} (left) and in Table~\ref{tab:3slits-g}. 

We see from Table~\ref{tab:3slits-g} that the middle segment $[a,b]$ on the real axis has no effect on the value of the capacity for this case. 
However, this will not be the case if we move the middle segment away from the real axis. 
To show that, we keep $E$ and $\delta$ the same as above and we change the domain $B$ to $B=\CC\backslash[a+\i,b+\i]$, i.e., we move the middle segment vertically by unity. Then, the values of the capacity depends on $a$ and $b$ (see the fifth column in Table~\ref{tab:3slits-g}). The level curves of the potential function are presented in Figure~\ref{fig:3slits-lc-g} (right).

\begin{figure}[ht] %
\centerline{
\scalebox{0.55}{\includegraphics[trim=0 1cm 0cm 1cm,clip]{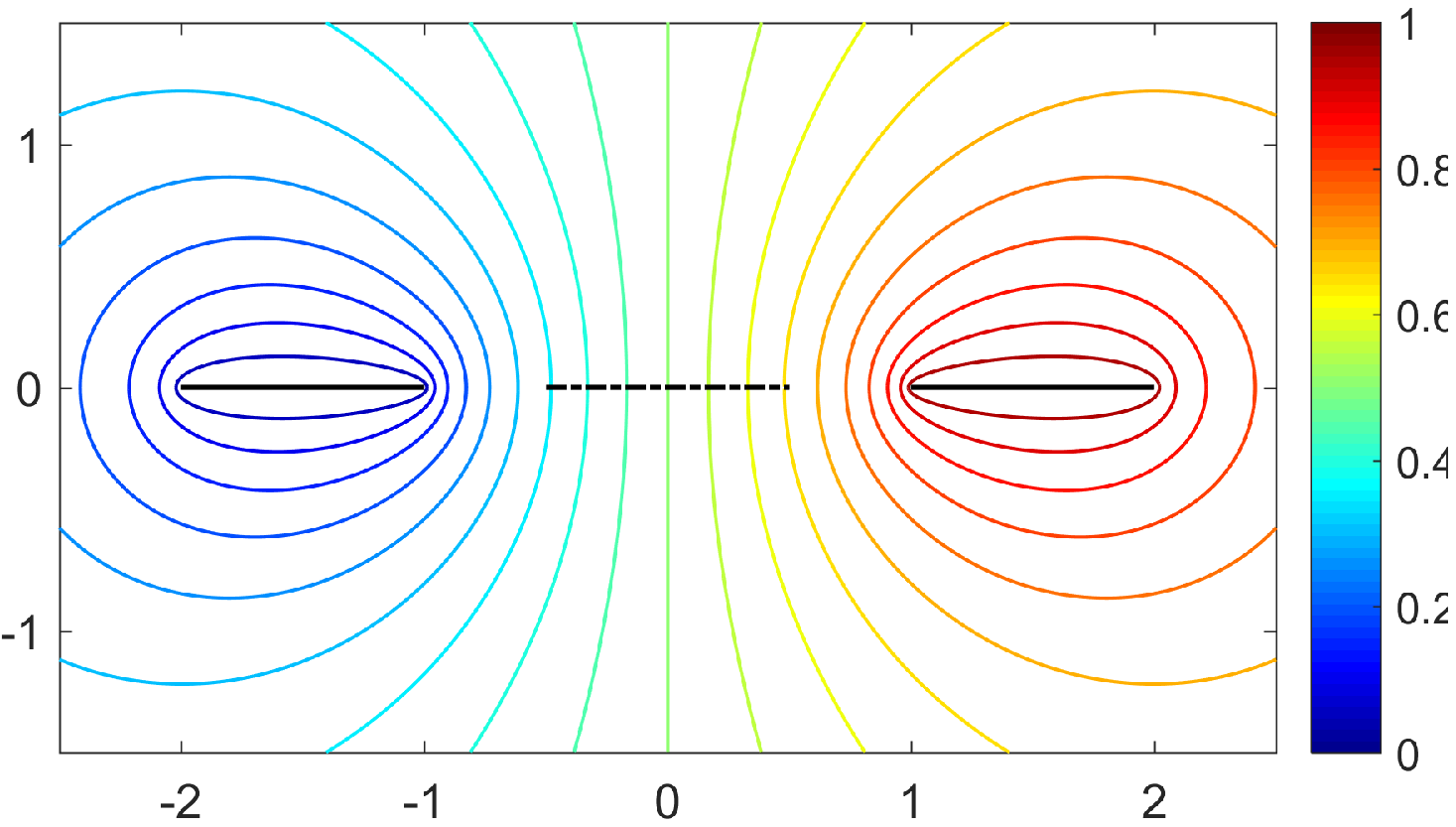}}
\hfill
\scalebox{0.55}{\includegraphics[trim=0 1cm 0cm 1cm,clip]{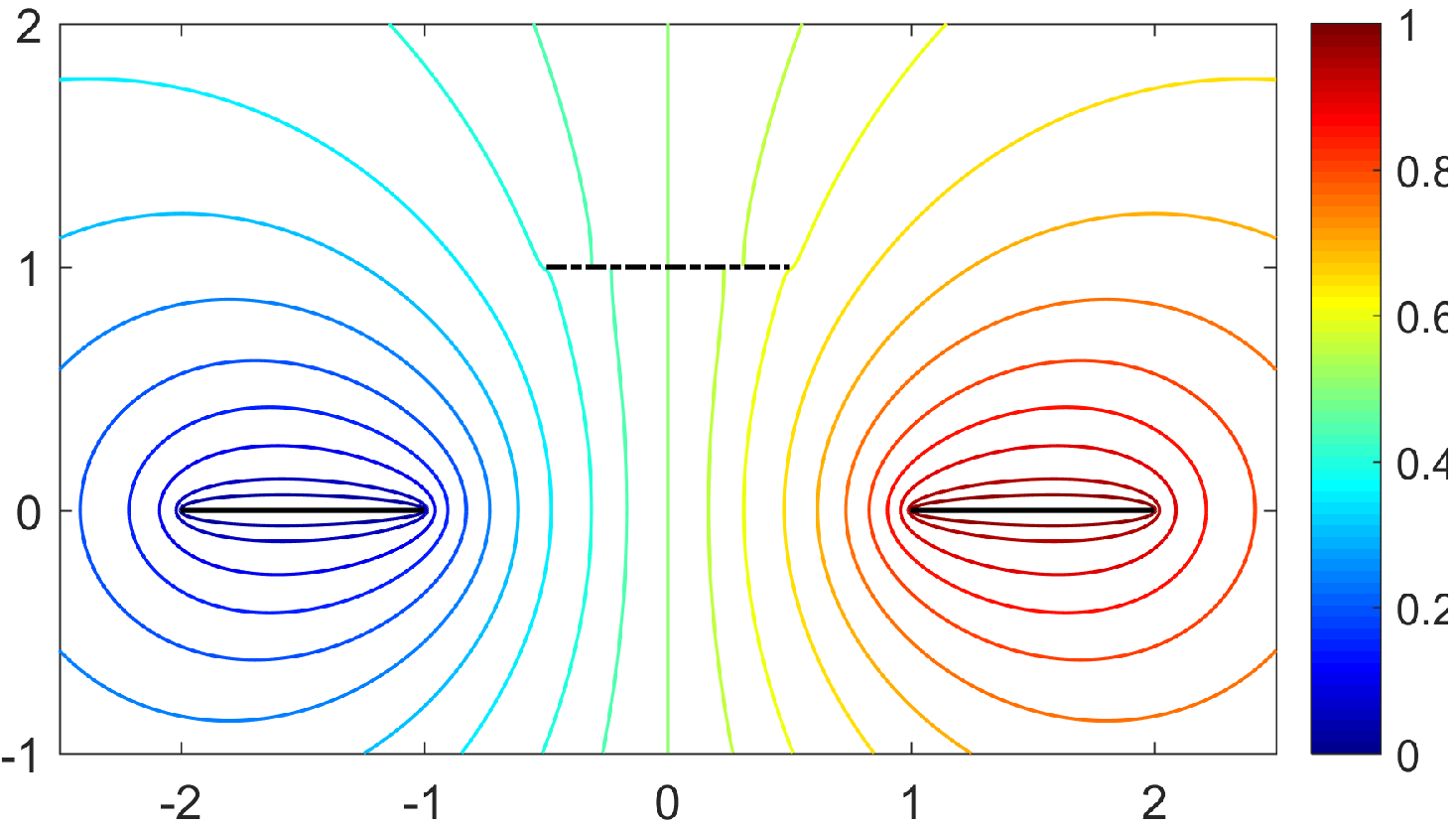}}
}
\caption{The level curves of the function $u$ for the condenser in Example~\ref{ex:3slits-g} for $B=\CC\backslash[a,b]$ (left) and $B=\CC\backslash[a+\i,b+\i]$ (right).}
\label{fig:3slits-lc-g}
\end{figure}

\begin{table}[ht]
\caption{The approximate values of the capacity $\capp(C)$ for Example~\ref{ex:3slits-g}.}
\label{tab:3slits-g}%
\vskip-0.5cm
\[
\begin{array}{r@{\hspace{0.5cm}}l@{\hspace{0.5cm}}l@{\hspace{0.5cm}}l@{\hspace{0.5cm}}l} \hline %
a      & b     & c   & B=\CC\backslash[a,b] & B=\CC\backslash[a+\i,b+\i] \\ \hline %
-0.9   & 0     & 2   & 1.279261571170975    & 1.276631670192704 \\
-0.5   & 0.5   & 2   & 1.279261571170975    & 1.278826082326995 \\
-0.9   & 0.9   & 2   & 1.279261571170975    & 1.274579061435374\\
 0     & 0.9   & 2   & 1.279261571170975    & 1.276631670192704\\
-0.5   & 0.5   & 3   & 1.563401922696102    & 1.563011913331686\\
-0.7   & 0.2   & 3   & 1.563401922696101    & 1.562502672069208\\
 0.5   & 0.8   & 3   & 1.563401922696093    & 1.562906600728627\\
\hline %
\end{array}
\]
\end{table}

\nonsec{\bf Cantor set.}\label{ex:cant-s}

In Example~\ref{ex:cant-d}, we consider the Cantor dust which a generalization of the classical Cantor middle third set to dimension two. The boundaries of the closed sets $S_k$ in Example~\ref{ex:cant-d} were piecewise smooth Jordan curves so the method presented in Section~\ref{sc:num} is directly applicable to the problem considered in Example~\ref{ex:cant-d}. In this example, we consider the classical Cantor middle third set which means the domain $G$ is bordered by slits and hence the presented method is not directly applicable. However, the presented method can be used with the help of conformal mappings as explained in Example~\ref{ex:3slits}.

Let $I_k$, $k=0,1,2,\ldots$, be as defined in Example~\ref{ex:cant-d}. Then, the classical Cantor middle third set is defined as 
\[
I = \bigcap_{k=1}^{\infty} I_k.
\]
For $k=0,1,2,\ldots$, the closed set $I_k$ consists of $2^k$ closed intervals $E_1,E_2,\ldots,E_{2^k}$ (see Figure~\ref{fig:cant-s-lc} for $k=2$ (left) and $k=3$ (right)).
We consider the generalized condensers $C_k=(B,E,\delta)$ with $B=\CC$ and $E=\{E_1,E_2,\ldots,E_{2^k}\}$. For the levels of the potential function $\delta=\{\delta_j\}_{j=1}^{4^k}$, we assume $\delta_j=0$ for half of the plates (the plates on the left of the line $x=0.5$) and $\delta_j=1$ for the other half (the plates on the right of the line $x=0.5$). 
Thus, $\ell=0$, $m'=m=2^k$, and the generalized condenser reduces to a regular condenser. The field of the condenser, $G$, is then the unbounded multiply connected domain in the exterior of the closed sets $E_k$ (see Figure~\ref{fig:cant-s-lc}).

The approximate values of the capacity for $k=1,2,\ldots,9$ are shown in Table~\ref{tab:cant-s} and the level curves of the function $u$ for $k=2,3$ are shown in Figure~\ref{fig:cant-s-lc}. For each $k$, we need first to use the iterative method presented in~\cite{NG18} to compute a domain $\hat G$ bordered by smooth Jordan curves which is conformally equivalent to the domain $G$. Then, we use the presented method for the new domain $\hat G$ and the method requires solving $m=2^k$ integral equations. The total CPU time for the two steps for each $k$ is presented in Table~\ref{tab:cant-s}. The presented numerical results obtained with $n=2^{10}$.

\begin{table}[ht]
\caption{The approximate values of the capacity $\capp(C_k)$ for Example~\ref{ex:cant-s}.}
\label{tab:cant-s}%
\vskip-0.5cm
\[
\begin{array}{l@{\hspace{1.00cm}}l@{\hspace{1.00cm}}l@{\hspace{1.00cm}}l} \hline %
k    & m=2^k    & \capp(C_k)          & {\rm Time\;(sec)} \\ \hline %
1    & 2        & 1.563401922696121   & 0.22     \\
2    & 4        & 1.521894262663735   & 0.70     \\
3    & 8        & 1.498986233451143   & 2.17     \\
4    & 16       & 1.487078427246902   & 6.84     \\
5    & 32       & 1.480987379910617   & 23.20    \\
6    & 64       & 1.477880583227881   & 91.26    \\
7    & 128      & 1.476295519723391   & 371.64   \\
8    & 256      & 1.475486275937497   & 1405.52  \\
9    & 512      & 1.475072901005890   & 6158.92  \\
\hline %
\end{array}
\]
\end{table}

\begin{figure}[ht] %
\centerline{
\scalebox{0.55}{\includegraphics[trim=0 3cm 0cm 3cm,clip]{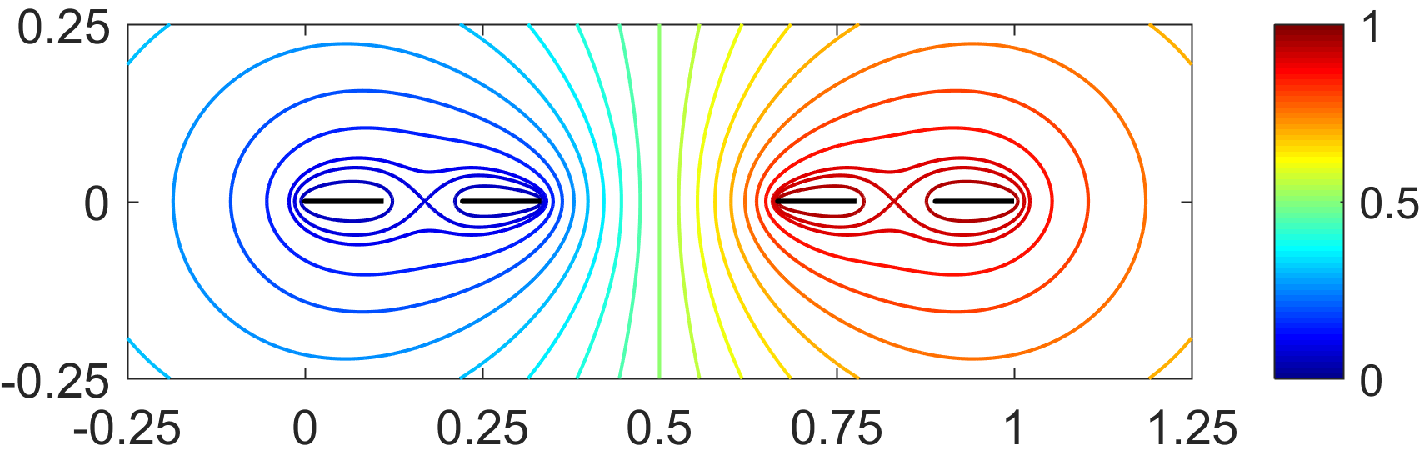}}
\hfill
\scalebox{0.55}{\includegraphics[trim=0 3cm 0cm 3cm,clip]{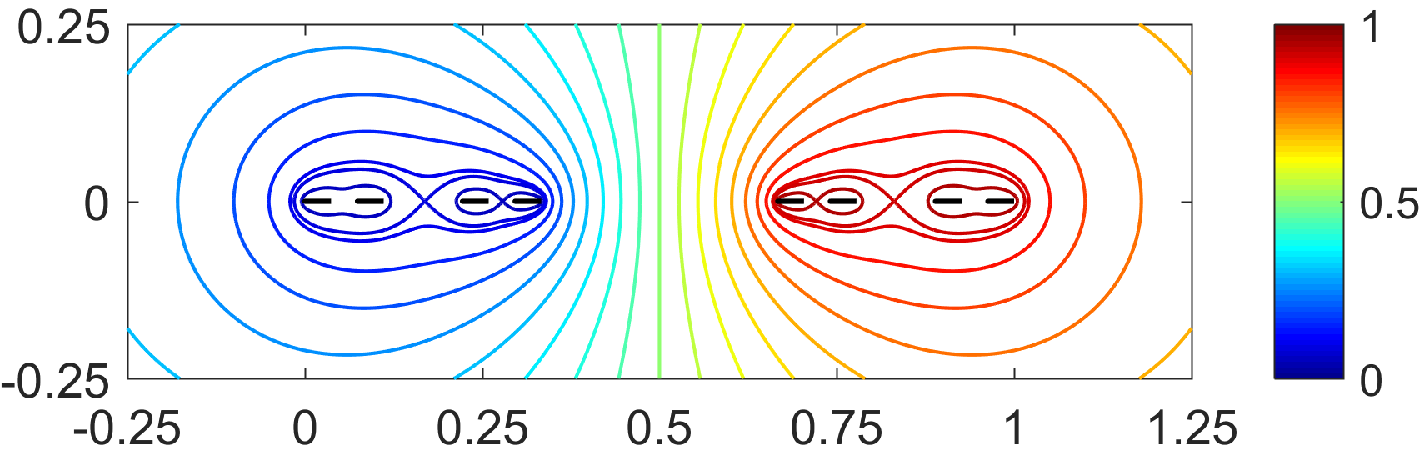}}
}
\caption{The level curves of the function $u$ for the condenser in Example~\ref{ex:cant-s} for $k=2$ (left) and $k=3$ (right).}
\label{fig:cant-s-lc}
\end{figure}

% FILE:  gcsec601.tex
%%%%%%%%%%%%%%%%%%%%%%%%%%%%%%%%%%%%
%%%%%%%%%%%%%%%%%%%%%%%%%%%%%%%%%%%%
%%%%%%%%%%%%%%%%%%%%%%%%%%%%%%%%%%
%%%%%%%%%%%%%%%%%%%%%%%%%%%%%%%%%%
\section{Harmonic measure}

Assume that the multiply connected domain $G$ is as described in Section~\ref{sec:u} with $\ell=0$, i.e., $G$ is a multiply connected domain of connectivity $m$ bordered by $\Gamma=\cup_{k=1}^{m}\Gamma_k$ where $\Gamma_m$ is the external boundary component if $G$ is bounded. In this section, we shall use the method described above to compute the ``harmonic measure'' for the multiply connected domain $G$. 

For a fixed $j$, $j=1,2,\ldots,m$, let $u$ be the harmonic function in $G$ that satisfy the boundary condition 
\begin{equation}\label{eq:har-u-bd}
u(\zeta)=
\begin{cases} 
1, & \zeta \in \Gamma_j, \\ 
0, & \zeta \in \Gamma_k, \quad k\ne j, \quad k=1,2,\ldots,m,
\end{cases}
\end{equation}
where $u$ is assumed to be bounded at $\infty$ for unbounded $G$.
Then the function $u$ is called the harmonic measure of $\Gamma_j$ with respect to $G$ and will be denoted by $\omega_{G,\Gamma_j}$~\cite{Cro-Green,gm,Kra,Tsu75}. From the Maximum Principle for harmonic functions~\cite[p.~77]{Tsu75} it follows that $0<\omega_{G,\Gamma_j}(z)<1$ for $z\in G$. The harmonic measure $\omega_{G,\Gamma_j}(z)$ is invariant under conformal maps. If $\Phi$ is a conformal mapping from the domain $G$ onto $\Phi(G)$, then~\cite{avv,gm}
\[
\omega_{G,\Gamma_j}(z)=\omega_{\Phi(G),\Phi(\Gamma_j)}(\Phi(z))
\]
for all $z\in G$, $j=1,2,\ldots,m$.

The boundary condition~\eqref{eq:har-u-bd} is a special case of the boundary condition~\eqref{eq:mix-bd} (here, $\ell=0$ so we will not have the normal derivative boundary condition). Thus, the algorithm presented in Section~\ref{sc:algorithm} can be used to compute the harmonic measure $\omega_{G,\Gamma_j}(z)$ for $z\in G$, $j=1,2,\ldots,m$. In fact, by the definition of the function $\delta$ in Example~\ref{ex:3slits}, the level curves presented in Figure~\ref{fig:3slits-lc} (right) are the level curves of the harmonic measure $\omega_{G,\Gamma_2}(z)$ for the triply connected domain $G$ in the exterior of the three slits $\Gamma_1$ (the left slit),  $\Gamma_2$ (the middle slit), and $\Gamma_3$ (the right slit). The level curves presented in Figure~\ref{fig:3slits-lc} (left) are the level curves of the harmonic measure $\omega_{G,\Gamma_1}(z)+\omega_{G,\Gamma_2}(z)$.
 
We consider two more examples as following.

\nonsec{\bf Annulus.}\label{ex:hm-ann}

Let $G$ by the annulus $G= \{z \in \mathbb{C}: q<|z|<1\}$. Then the exact harmonic measures of the inner circle $\Gamma_1= \{z \in \mathbb{C}: |z|=q\}$ and the outer circle $\Gamma_2= \{z \in \mathbb{C}: |z|=1\}$ with respect to $G$ are given by
\[
\omega_{G,\Gamma_1}(z) =   \frac{\log|z|}{\log q}, \quad
\omega_{G,\Gamma_2}(z) = 1-\frac{\log|z|}{\log q}, \quad z\in G.
\]

We use the method presented in Section~\ref{sc:num} with $n=2^{10}$ to compute approximate values of the harmonic measures $\omega_{G,\Gamma_1}(z)$ and $\omega_{G,\Gamma_2}(z)$ for $z\in G$. The absolute error in the computed values are shown in Figure~\ref{fig:hm-ann}.

\begin{figure}[ht] %
\centerline{
\scalebox{0.5}[0.5]{\includegraphics[trim=0 0cm 0cm 0cm,clip]{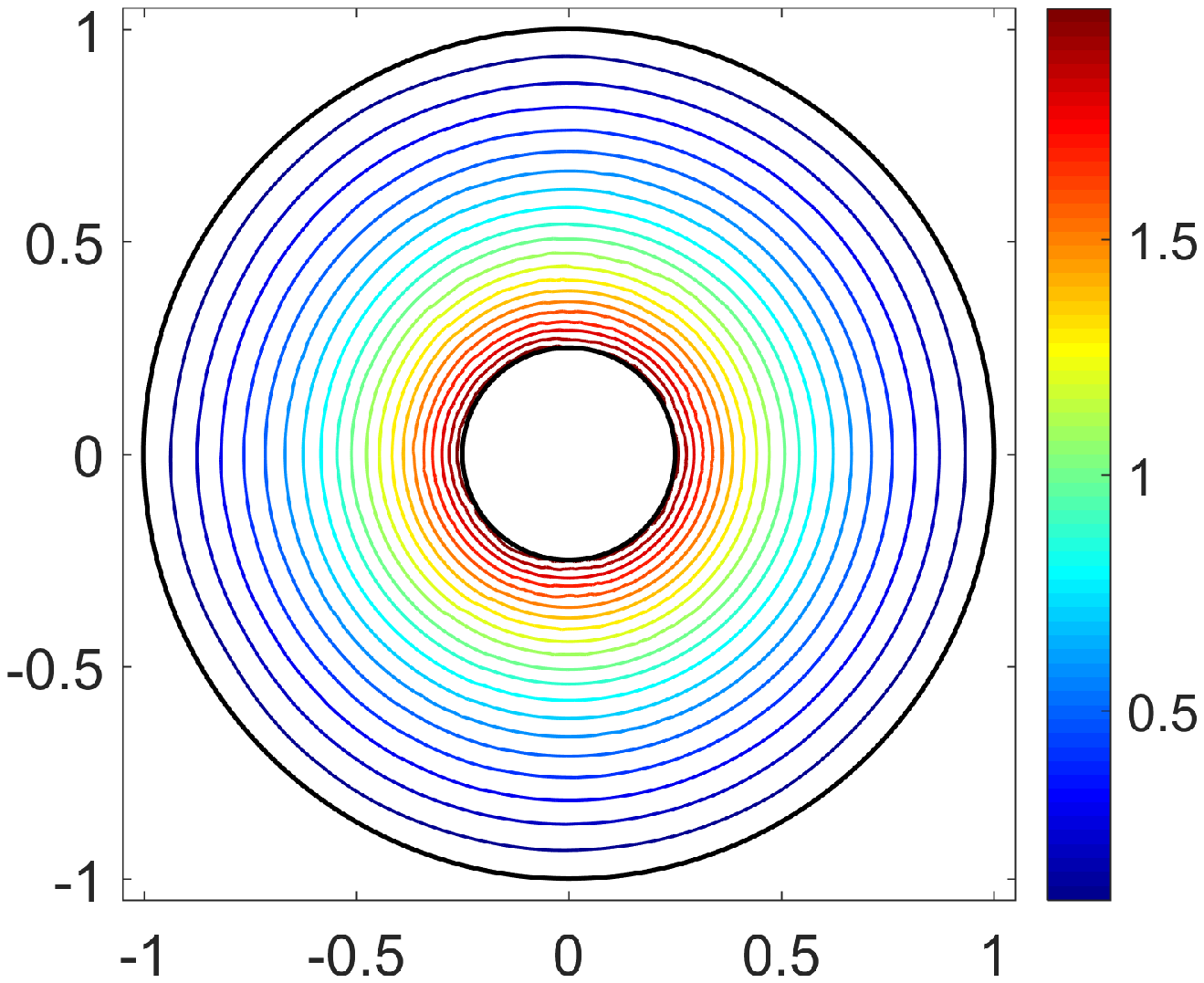}}
\hfill
\scalebox{0.5}[0.5]{\includegraphics[trim=0 0cm 0cm 0cm,clip]{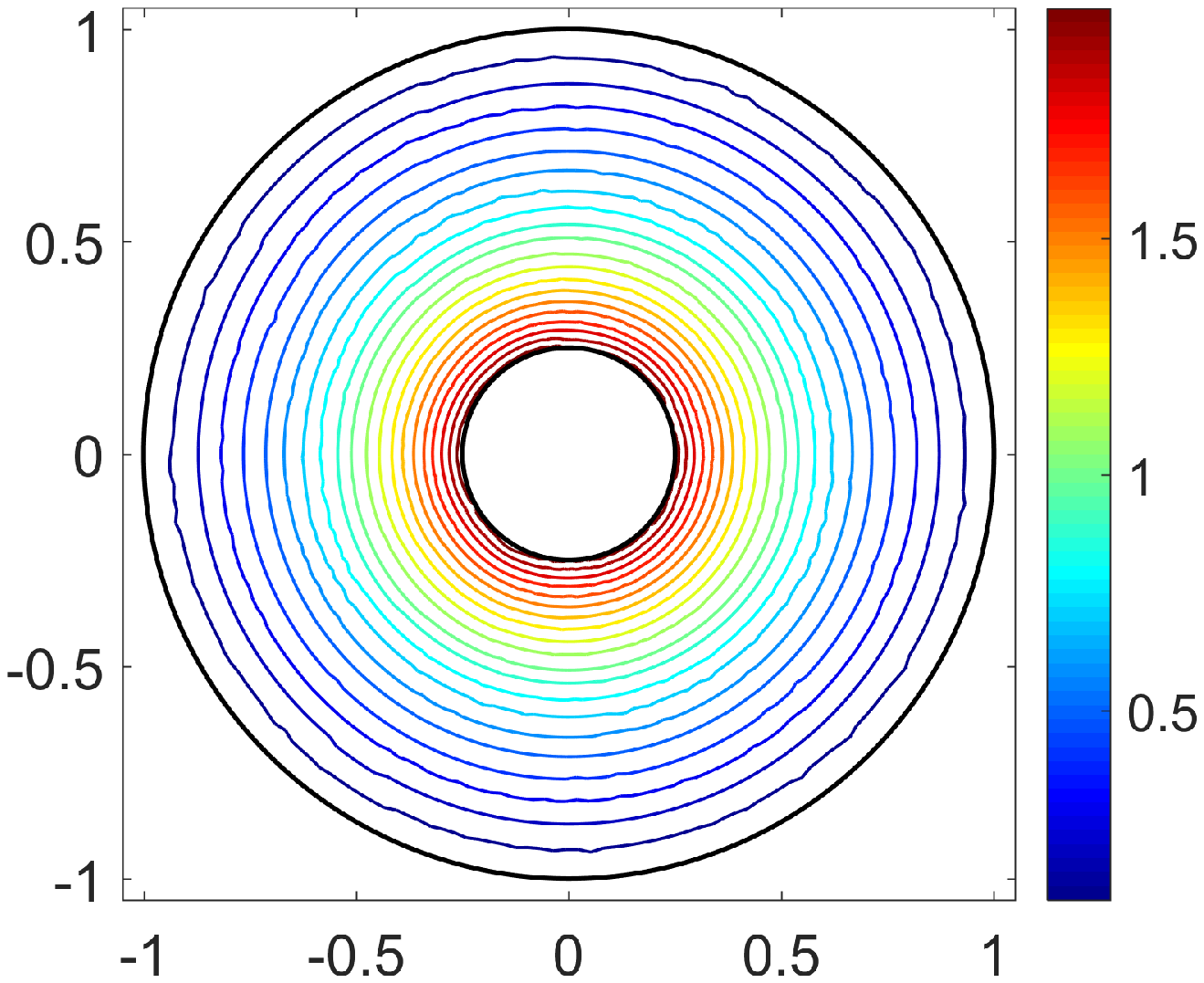}}
}
\caption{The level curves of the absolute error in the computed values of the harmonic measures $\omega_{G,\Gamma_1}(z)$ (left) and $\omega_{G,\Gamma_2}(z)$ (right) for Example~\ref{ex:hm-ann}.}
\label{fig:hm-ann}
\end{figure}

\nonsec{\bf Two disks and two polygons.}\label{ex:hm-plg}

We consider the multiply connected domain $G$ of connectivity $4$ in the exterior of the curves $\Gamma_1,\Gamma_2,\Gamma_3$ and in the interior of the curve $\Gamma_4$. Here, $\Gamma_1$ is the circle $|z-0.5|=0.25$, $\Gamma_2$ is the circle $|z+0.5|=0.25$, $\Gamma_3$ is the polygon with the vertices $0.5-0.5\i,0.5-0.8\i,-0.5-0.8\i,-0.5-0.5\i$, and $\Gamma_4$ is the polygon with the vertices $1,\i,-1,-1-\i,1-\i$. 
We use the method presented in Section~\ref{sc:num} with $n=5\times2^{8}$ to compute approximate values of the harmonic measures $\omega_{G,\Gamma_1}(z)$, $\omega_{G,\Gamma_2}(z)$, $\omega_{G,\Gamma_3}(z)$ and $\omega_{G,\Gamma_4}(z)$ for $z\in G$. The level curves of the computed harmonic measures are shown in Figure~\ref{fig:hm-plg}.

\begin{figure}[ht] %
\centerline{
\scalebox{0.5}[0.5]{\includegraphics[trim=0 0cm 0cm 0cm,clip]{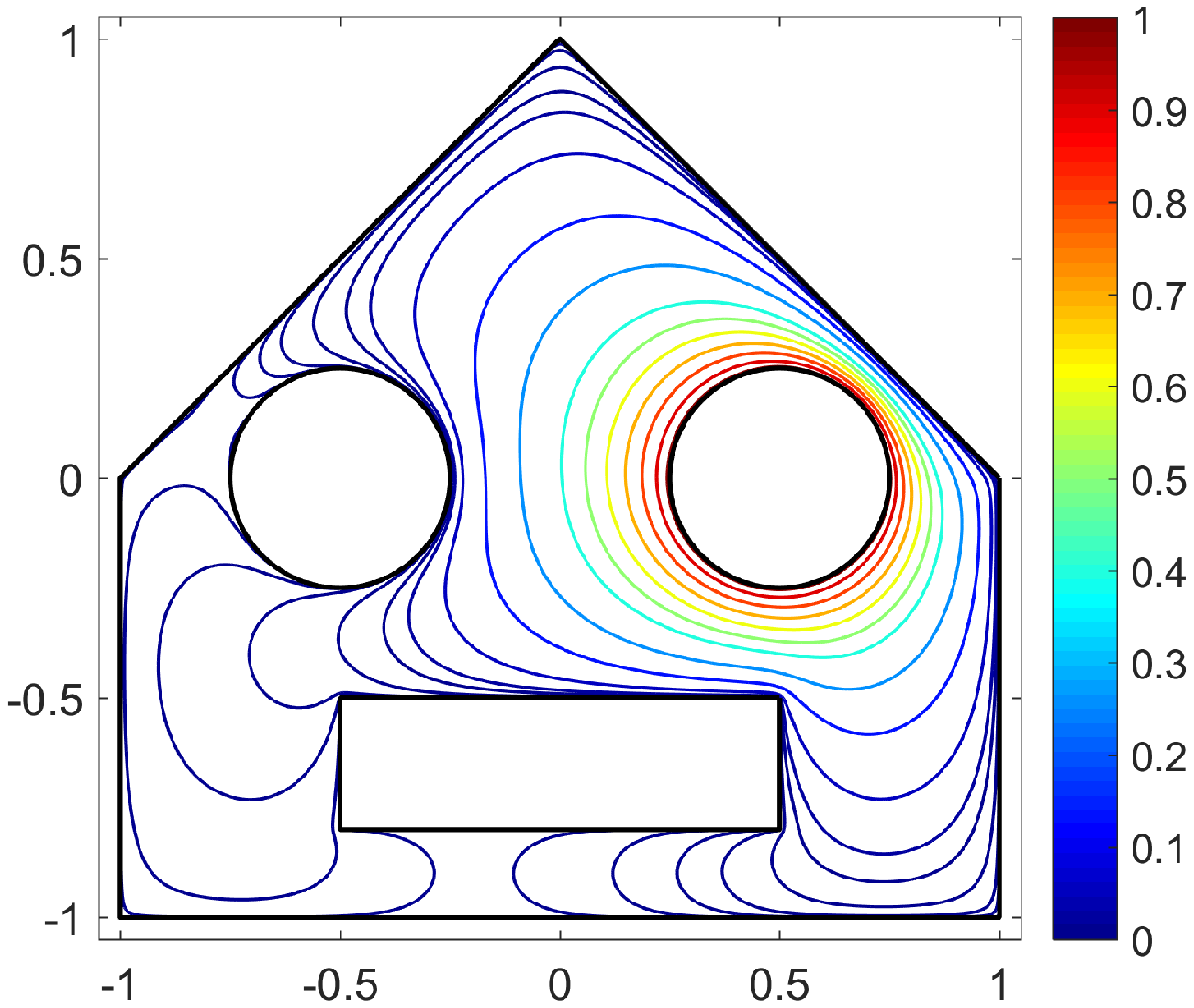}}
\hfill
\scalebox{0.5}[0.5]{\includegraphics[trim=0 0cm 0cm 0cm,clip]{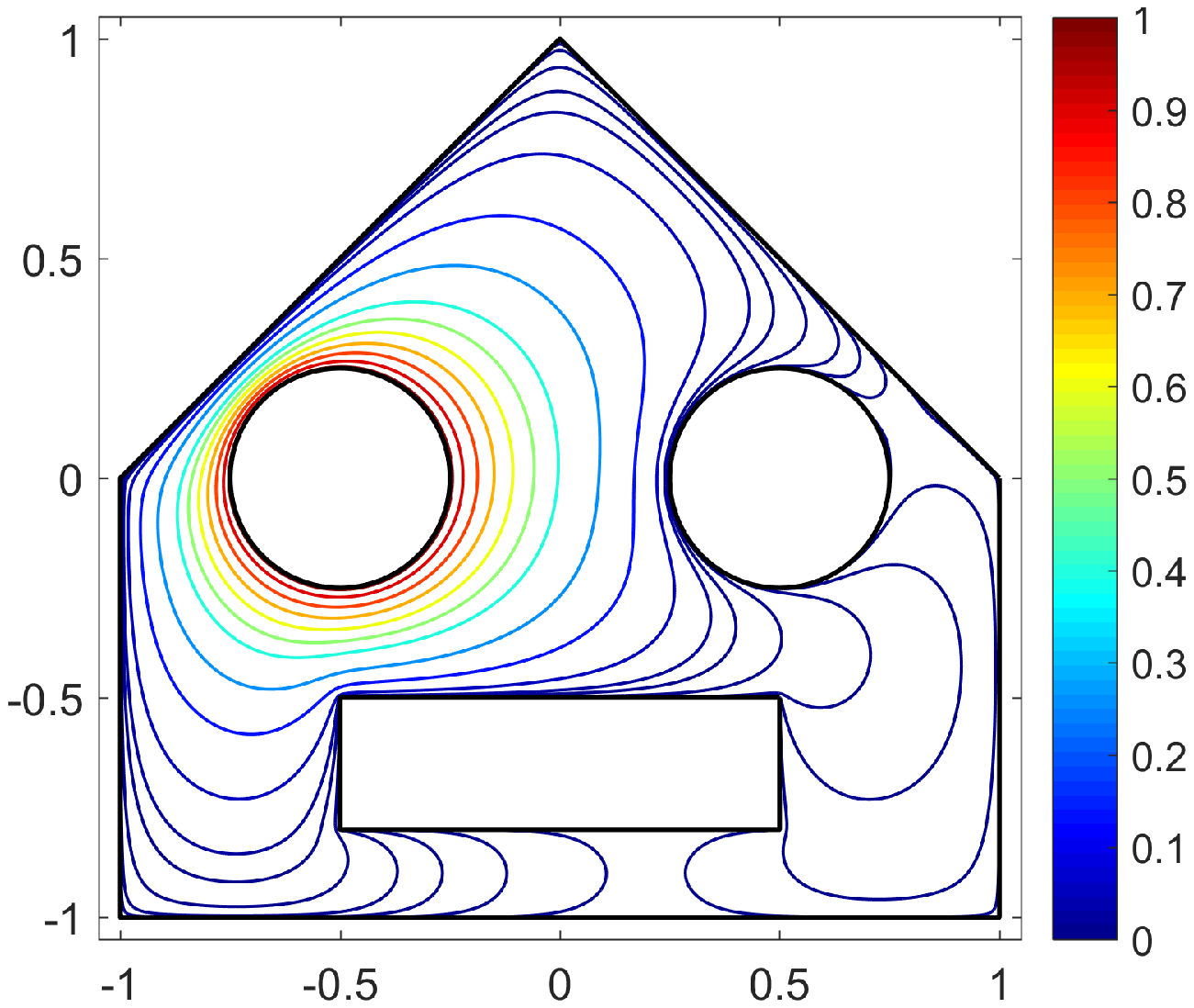}}
}
\centerline{
\scalebox{0.5}[0.5]{\includegraphics[trim=0 0cm 0cm 0cm,clip]{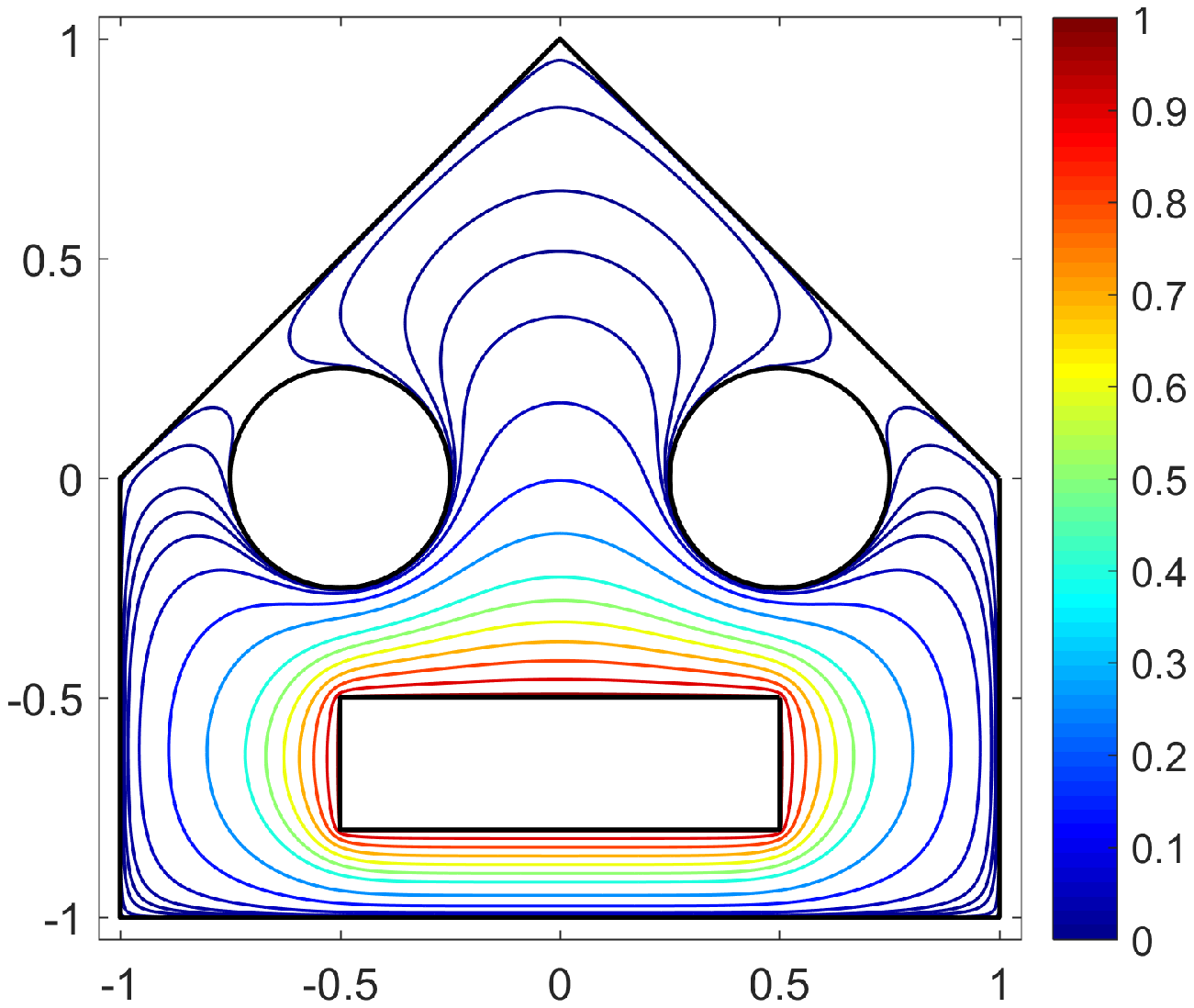}}
\hfill
\scalebox{0.5}[0.5]{\includegraphics[trim=0 0cm 0cm 0cm,clip]{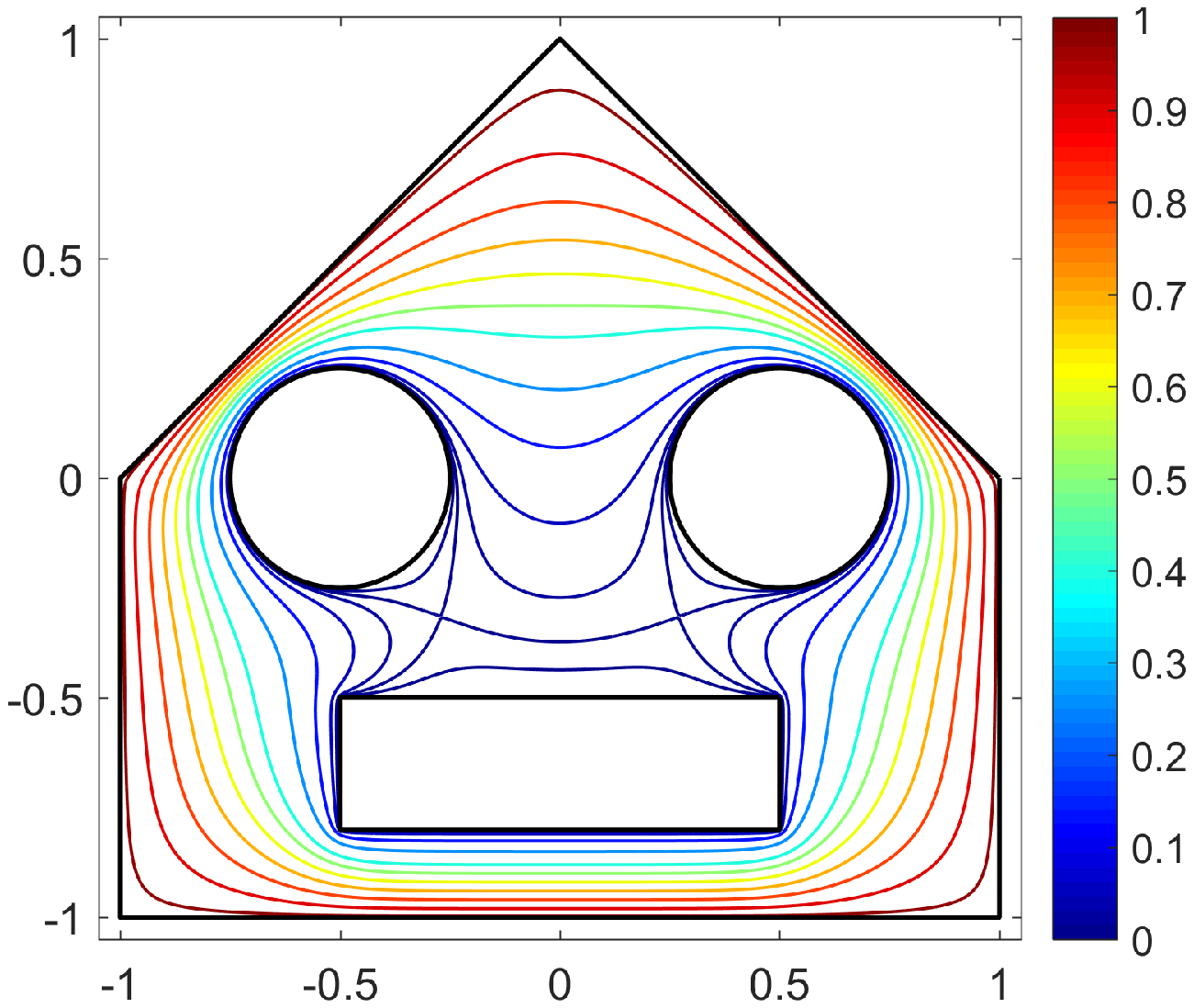}}
}
\caption{The level curves of the computed harmonic measures $\omega_{G,\Gamma_1}(z)$ (top, left), $\omega_{G,\Gamma_2}(z)$ (top, right), $\omega_{G,\Gamma_3}(z)$ (bottom, left), and $\omega_{G,\Gamma_4}(z)$ (bottom, right) for Example~\ref{ex:hm-plg}.}
\label{fig:hm-plg}
\end{figure}

\bigskip

%\textbf{Acknowledgments.}
%
%This work was partially supported by ....
%Unfinished........

%%%%%%%%%%%%%%%%%%
%%%%%%%%%%%%%%%%%%
%%%%%%%%%%%%%%%%%%
% FILE: gcbiblio101.tex
%%%%%%%%%%%%%%%%%%
%%%%%%%%%%%%%%%%%%
%%%%%%%%%%%%%%%%%%

%%%%%%%%%%%%%%%%%%%%%%%%%%%%%%%%%%%%%%%%%%%%%
%\nocite{aa}
%\nocite{pac}
%\nocite{hkvz}

%\bibliographystyle{siamplain}
%\bibliography{Bref}

\end{document}